\documentclass[a4paper,10pt]{amsart}

\usepackage[utf8]{inputenc}
\usepackage[english,ngerman]{babel}
\usepackage{array,tikz,listings,color,graphicx,csquotes,libertine,datetime,enumitem,amssymb}
\usepackage[hypertexnames=false]{hyperref}
\usetikzlibrary{shapes.misc} 
\usepackage{pgfplots} \pgfplotsset{compat=1.14}
\usepackage{wrapfig,tikzscale}

\usepackage{colortbl}

\newtheorem{definition}{Definition}
\newtheorem{lemma}{Lemma}
\newtheorem{theorem}{Theorem}
\newtheorem{series}{Series}

\usepackage{fancyhdr}
\makeatletter
  \def\section{\@startsection{section}{1}%
    \z@{.7\linespacing\@plus\linespacing}{.6\linespacing}%
    {\Large\normalfont\scshape\bfseries\centering}}
\makeatother

  \usepackage{hyperref,tikz,url,amsmath,amssymb}
  \usepackage{pgfplots}\pgfplotsset{compat=1.17}
  \usepackage{numprint}  \npthousandsep{\hspace{1.3pt}}  \npdecimalsign{.}
  \newcommand{\nfrac}[2]{\frac{\numprint{#1}}{\numprint{#2}}}
  \newcommand{\tnfrac}[2]{\tfrac{\numprint{#1}}{\numprint{#2}}}

\usepackage{breqn,arydshln}

\begin{document}
\selectlanguage{english}
    \title{\Large{Asymptotic expansions\\
    for the truncation error\\
    in Ramanujan-type series}}
    \author{Lorenz Milla, 05/2022}

\maketitle

\begin{abstract}

Many of the fastest known algorithms to compute $\pi$ involve generalized hypergeometric series, such as the Ramanujan-Sato series. In this paper, we investigate the rates of convergence for several such series and we give asymptotic expansions for the error of finite approximation.

For example, when using the first $n$ terms of the Chudnovskys' series, we obtain the finite approximation $\pi_n\approx \pi$. It is known that the truncation error satisfies $|\pi_n-\pi|\approx 53360^{-3n}.$ In this paper, we prove that the asymptotic expansion for the truncation error in the Chudnovskys' series is
$$\left|\pi_n-\pi\right|=53360^{-3n}\cdot\frac{\numprint{106720}\sqrt{\numprint{10005}\pi}}{\numprint{1672209}\sqrt{n}}\cdot\exp\left(\frac{A_1}{n}+\frac{A_2}{n^2}+\frac{\delta_n}{n^3}\right),$$
with $\numprint{0.006907}<\delta_n<\numprint{0.008429}$ and the exact rational values of $A_1$ and $A_2$:
\begin{align*}
 A_1&= -\nfrac{1781843197433}{7456754505816},\\
 A_2&=  -\nfrac{1080096011925710088395}{3475199235000451148614116}.
\end{align*}
Thus we demonstrate how to establish precise error bounds for the approximations for $\pi$ obtained through Ramanujan-like series for $1/\pi$.

We also give asymptotic expansions for all known rational hypergeometric series for $1/\pi$ in the appendix.
\bigskip

\bigskip


\end{abstract}

\thispagestyle{empty}
  \section*{Introduction}\label{secIntro}
  \renewcommand{\leftmark}{Introduction}
  \begin{subequations}
In 1914, Srinivasa Ramanujan gave 17 series for the number $\pi$. The fastest among them (see \cite[Eq. (44)]{rama1914}) is
\begin{align}
    \frac{1}{\pi }&=\frac{\sqrt{8}}{9801}\sum_{k=0}^{\infty}\frac{(4k)!}{(k!)^4}\,\frac{1103+26390\,k}{396^{4k}}.\label{rama396}
\end{align}
The brothers David and Gregory Chudnovsky found an even faster series in 1988 (see \cite[Eq. (1.5)]{chud1988}),
which was used in most recent computations of $\pi$ -- it reads
\begin{align}
    \frac{\sqrt{640320^3}}{12\pi}=\sum_{k=0}^{\infty}
\frac{(-1)^k\left(6k\right)!}{\left(3k\right)!\left(k!\right)^3}\,\frac{13591409 + 545140134\,k}{640320^{3k}}.\label{chudseries}
\end{align}
More details can be found in \cite{Berndt}.
In \cite{MillaArxiv}, we gave a detailed proof of this series using elliptic curves and the Picard Fuchs differential equation.
In \cite{millaRamJ}, we calculated the coefficients in the Chudnovskys' series using the integrality of certain non-holomorphic modular functions.
\end{subequations}


To calculate digits of $\pi$, one has to decide how many terms to use:
\begin{definition}\label{defi1}
The first $n$ terms of the Chudnovskys' series (\ref{chudseries}) yield the finite approximation $\pi_n\approx \pi$
\begin{subequations}
\begin{align}
    \frac{1}{\pi_n} &= \frac{12\cdot \numprint{545140134}}{\sqrt{\numprint{640320}^3}}\,\sum_{k=0}^{n-1}(-1)^k s_k,\label{eqdefipin}\\
    \text{with}\qquad
    s_k &= \frac{\left(6k\right)!}{\left(3k\right)!\left(k!\right)^3}\,\frac{k+S}{\numprint{640320}^{3k}}\label{eqdefisk}\\
    \text{and}\qquad S &=\nfrac{13591409}{545140134}.\label{eqdefiS}
\end{align}
Throughout the paper we denote the convergence rate of $s_k$ by
\begin{align}
    \varepsilon=\frac{6^6 / 3^3}{\numprint{640320}^{3}}=\frac{\numprint{1728}}{\numprint{640320}^{3}}=\numprint{53360}^{-3}.\label{eqdefieps}
\end{align}
\end{subequations}
\end{definition}
In Section \ref{secttheo1}, we prove the following asymptotic expansion of the truncation error in the Chudnovskys' series:

\begin{theorem}\label{theo01} 
For all $n\geq 1$, the omitted terms of the Chudnovskys' series (\ref{chudseries}) admit the asymptotic expansion
\begin{subequations}
\begin{align}
    \left|\sum_{k=n}^\infty (-1)^k s_k\right|&=\frac{s_n}{1+\varepsilon}
    \left(1 + \frac{a_1}{n} + \frac{a_2}{n^2}+\frac{e_n\varepsilon}{n^3}\right).\label{th1expans}
\end{align}
For the definition of $s_k$, $S$ and $\varepsilon$ see (\ref{eqdefisk}), (\ref{eqdefiS}) and (\ref{eqdefieps}).
The coefficients of this expansion are the rational numbers
\begin{align}
a_1 &= \frac{\varepsilon}{2(1+\varepsilon)} = \nfrac{1}{303862746112002}\label{eqa1}\\
\text{and}\qquad a_2 &=\frac{S\varepsilon}{1+\varepsilon}-\frac{ 23\varepsilon-4\varepsilon^2}{36 (1+\varepsilon)^2} = \nfrac{-62186213362465}{15388761412454497761254741334}.\label{eqa2}
\end{align}
The error term $e_n$ in this expansion satisfies
\begin{align}
    0.3216 < e_n < 0.6704.\label{th1errorbound}
\end{align}
\end{subequations}
\end{theorem}

In Section \ref{secrekurs}, we derive the recursion formula (\ref{eqrekurs}) which allows one to compute the rational coefficients $a_k$ of the asymptotic expansion (\ref{th1expans}) more efficiently and to higher orders.
But this involves a formal power series (\ref{defipoweran}) which is divergent, thus the results from Section \ref{secrekurs} do not readily yield rigorous error bounds like (\ref{th1errorbound}) in Theorem~\ref{theo01}.

In Section \ref{secttheo2}, we use Stirling's approximation and Theorem \ref{theo01} to prove the following asymptotic expansion for $|\pi_n-\pi|$ mentioned in the abstract:
\begin{theorem} \label{theo02}
If we denote by $\pi_n$ the approximation we obtain by using the first $n$ terms in the Chudnovskys' series as in (\ref{eqdefipin}), then we have the asymptotic expansion
\begin{subequations}
\begin{align}
\left|\pi_n-\pi\right|=53360^{-3n}\cdot\frac{A_0}{\sqrt{n}}\cdot\exp\left(\frac{A_1}{n}+\frac{A_2}{n^2}+\frac{\delta_n}{n^3}\right).
\end{align}
Here, the error term $\delta_n$ satisfies
\begin{align}
    \numprint{0.006907}<\delta_n<\numprint{0.008429}
\end{align}
for all $n\geq 1$ and the coefficients $A_0$, $A_1$ and $A_2$ can be expressed in terms of $S$ from (\ref{eqdefiS}), $\varepsilon$ from (\ref{eqdefieps}), $a_1$ from (\ref{eqa1}) and $a_2$ from (\ref{eqa2}):
\begin{align}
    A_0 &=\frac{6\sqrt{\pi}\cdot \numprint{545140134}}{\sqrt{\numprint{640320}^3}(1+\varepsilon)} = \frac{\numprint{106720}\cdot\sqrt{\numprint{10005}\pi}}{\numprint{1672209}},\\
    A_1&=a_1 -\frac{19}{72}+S  = -\nfrac{1781843197433}{7456754505816},\\
    A_2&= a_2 - \frac{a_1^2}{2} -\frac{S^2}{2}  = -\nfrac{1080096011925710088395}{3475199235000451148614116}.
\end{align}
\end{subequations}
\end{theorem}

We have chosen to calculate all terms exactly up to order $n^{-2}$,
with upper and lower bounds on the remainders $r$ like $\frac{a}{n^3}<r(n)<\frac{b}{n^3}$.
All parts of this paper can be verified without specialized software.

In Section~\ref{secconcl}, we will do a numerical analysis showing that the bounds we have proven are close to optimal.
We will also generalize our results to other hypergeometric series and we will give an asymptotic expansion for Ramanujan's series~(\ref{rama396}).

In Appendix~\ref{appmorti}, we prove an asymptotic expansion of a product of Pochhammer symbols.
This enables us to give asymptotic expansions for all known rational hypergeometric series for $1/\pi$ in Appendix \ref{AppSeries}.

  \section{Proof of Theorem \ref{theo01}}\label{secttheo1}
  \renewcommand{\leftmark}{Proof of Theorem \ref{theo01}}
  In this section, we prove Theorem \ref{theo01} using only basic methods of analysis.
For the definition of $s_n$, $S$ and $\varepsilon$ keep (\ref{eqdefisk}), (\ref{eqdefiS}) and (\ref{eqdefieps}) from Definition \ref{defi1} in mind.
\begin{lemma}\label{lemOhne1}
For all $n\geq 1$ there is $\varphi_n$ with $0.05<\varphi_n<2.19$ and
$$\frac{s_{n}}{\varepsilon\cdot s_{n-1}}=\sum_{j=0}^\infty \frac{b_j}{n^j}
=1 -\frac{1}{2n}+\frac{\frac{5}{36}-S}{n^2}+\frac{\frac{5}{72}- \frac{S}{2}+S^2}{n^3} + \frac{\varphi_n}{n^4}.$$
If $n\geq 2$, we have that $\varphi_n<0.11$ and
$$\frac{s_{n}}{\varepsilon\cdot s_{n-1}}<\exp\left(-\frac{1}{2n}-\frac{S-\frac{1}{72}}{n^2}+\frac{0.67}{n^3}\right).$$
\end{lemma}
\begin{proof}
The definition of $s_n = \frac{\left(6n\right)!}{\left(3n\right)!\left(n!\right)^3}\,\frac{n+S}{\numprint{640320}^{3n}}$ in (\ref{eqdefisk}) yields for all $n\geq 1$:
\begin{align}
  \label{defisnn}\frac{s_{n}}{\varepsilon\cdot s_{n-1}} 
    &=\left(1-\frac{1}{6n}\right)\left(1-\frac{1}{2n}\right)\left(1-\frac{5}{6n}\right)\cdot \frac{S + n}{S + n - 1}\\
    &=\left(1-\frac{3}{2n}+\frac{23}{36n^2}-\frac{5}{72n^3}\right)\cdot \left(1+\sum_{k=0}^\infty\frac{(1-S)^k}{n^{k+1}}\right)\nonumber\\
    &=1 - \frac{1}{2 n} + \frac{\frac{5}{36} - S}{n^2} + \frac{ \frac{5}{72}- \frac{S}{2}+S^2}{n^3} + \frac{\varphi_n}{n^4},\nonumber\\
\text{with}\qquad
\frac{\varphi_n}{n^4} &= \sum_{k=4}^\infty\frac{(\frac{5}{72}-\frac{23S}{36}+\frac{3S^2}{2}-S^3)(1-S)^{k-4}}{n^k} 
    =\frac{\frac{5}{72}-\frac{23S}{36}+\frac{3S^2}{2}-S^3}{n^4(1-\frac{1-S}{n})}.\nonumber
\end{align}
This proves $0.05<\varphi_n<2.19$ for $n\geq 1$ and $\varphi_n<0.11$ for $n\geq 2$. 
Next, (\ref{defisnn}) yields
\begin{align*}
    \ln\left(\frac{s_{n}}{\varepsilon\cdot s_{n-1}} \right)&=\ln\left(1-\frac{1}{6n}\right)+\ln\left(1-\frac{1}{2n}\right)+\ln\left(1-\frac{5}{6n}\right)\\
    &~~~~+\ln\left(1+ \frac{S}{n}\right)-\ln\left(1-\frac{1-S}{n}\right).
\end{align*}
Using $\ln(1-x)=-\sum_{k=1}^\infty \frac{x^k}{k}$ we obtain:
\begin{align*}
    \ln\left(\frac{s_{n}}{\varepsilon\cdot s_{n-1}} \right)&=\sum_{k=1}^\infty\frac{\left(1-S\right)^k-\left(-S\right)^k-\left(\frac{1}{6}\right)^k-\left(\frac{1}{2}\right)^k-\left(\frac{5}{6}\right)^k}{k\cdot n^k}\\
    &<-\frac{1}{2n}-\frac{S-\frac{1}{72}}{n^2}+\sum_{k=3}^\infty\frac{1}{k\cdot n^k}.
\end{align*}
For $n\geq 2$ we have $\sum_{k=3}^\infty\frac{1}{k\cdot n^k}<\frac{1}{3n^3}\sum_{k=0}^\infty\frac{1}{2^k}=\frac{2}{3n^3}$ which proves the lemma.
\end{proof}

\begin{lemma}\label{lemqnk}
For $n\geq 1$ and $k\geq 1$ we have that
\begin{align*}
    \frac{s_{n+k}}{\varepsilon^k\cdot s_{n}} 
    &<1-\frac{k}{2n}+\frac{\frac{3}{8}k^2+(\frac{19}{72}-S)k}{n^2}+\frac{1.2k^6}{n^3}.
\end{align*}
\end{lemma}
\begin{proof} Lemma \ref{lemOhne1} yields
\begin{align*}
    \frac{s_{n+k}}{\varepsilon^k\cdot s_{n}} &=\prod_{j=1}^k \frac{s_{n+j}}{\varepsilon\cdot s_{n+j-1}}
    <\exp\left(\sum_{j=1}^k\left( -\frac{1}{2(n+j)}-\frac{S-\frac{1}{72}}{(n+j)^2}+\frac{0.67}{(n+j)^3}\right)\right).
\end{align*}
For $x<0$, we have that $\exp(x)<1+x+\frac{1}{2}x^2$. Since the sum is negative, we obtain
\begin{align*}
    \frac{s_{n+k}}{\varepsilon^k\cdot s_{n}} &<1 +\sum_{j=1}^k\left( -\frac{1}{2(n+j)}-\frac{S-\frac{1}{72}}{(n+j)^2}+\frac{0.67}{(n+j)^3}\right)\\
    &~~~~+\frac{1}{2}\left(\sum_{j=1}^k\left( \frac{1}{2(n+j)}+\frac{S-\frac{1}{72}}{(n+j)^2}-\frac{0.67}{(n+j)^3}\right)\right)^2.
\end{align*}
We estimate $\frac{1}{n}-\frac{j}{n^2}<\frac{1}{n+j}<\frac{1}{n}-\frac{j}{n^2}+\frac{j^2}{n^3}$ as well as $\frac{1}{n^2}-\frac{2j}{n^3}<\frac{1}{(n+j)^2}<\frac{1}{n^2}$ and obtain \belowdisplayskip=-12pt
\begin{align*}
    \frac{s_{n+k}}{\varepsilon^k\cdot s_{n}} &< 1 +\sum_{j=1}^k\left( -\frac{1}{2n}+\frac{\frac{j}{2}-S+\frac{1}{72}}{n^2}+\frac{2j(S-\frac{1}{72})+0.67}{n^3}\right)\\
    &~~~~+\frac{1}{2}\left(\sum_{j=1}^k\left( \frac{1}{2n}-\frac{j}{2n^2}+\frac{j^2}{2n^3}+\frac{S-\frac{1}{72}}{n^2}\right)\right)^2\\
    &=1-\frac{k}{2n}+\frac{\frac{k(k+1)}{4}-Sk+\frac{k}{72}}{n^2}+\frac{k(k+1)(S-\frac{1}{72})+0.67k}{n^3}\\
    &~~~~+\frac{1}{2}\left( \frac{k}{2n}-\frac{k(k+1)}{4n^2}+\frac{k (k+1) ( 2 k+1)}{12n^3}+\frac{k(S-\frac{1}{72})}{n^2}\right)^2\\
    &\leq 1-\frac{k}{2n}+\frac{\frac{1}{4}k^2+(\frac{19}{72}-S)k}{n^2}+\frac{0.7k^2}{n^3} +\frac{1}{2}\left( \frac{k}{2n}+\frac{k^3}{2n^3}\right)^2\\
    &\leq 1-\frac{k}{2n}+\frac{\frac{1}{4}k^2+(\frac{19}{72}-S)k}{n^2}+\frac{0.7k^2}{n^3} +\frac{k^2}{8n^2}+\frac{k^4}{4 n^4} +\frac{k^6}{8 n^6}\\
    &\leq 1-\frac{k}{2n}+\frac{\frac{3}{8}k^2+(\frac{19}{72}-S)k}{n^2}+\frac{1.2k^6}{n^3}.
\end{align*}
\end{proof}

\noindent\textbf{Proof of upper bound in Theorem \ref{theo01}}:
Since $s_n$ decreases monotonically to zero, we obtain
\begin{align}
    \left|\sum_{k=n}^\infty (-1)^k s_k\right| 
    &= \sum_{k=0}^\infty (s_{n+2k}-s_{n+2k+1})
    = \sum_{k=0}^\infty s_{n+2k}\cdot\left(1-\varepsilon\cdot\frac{s_{n+2k+1}}{\varepsilon\cdot s_{n+2k}}\right)\nonumber\\
    &=s_n\cdot\sum_{k=0}^\infty \varepsilon^{2k}\cdot\underbrace{\frac{s_{n+2k}}{\varepsilon^{2k}\cdot s_n}}_{=Q_{n,k}}\cdot\underbrace{\left(1-\varepsilon\cdot\frac{s_{n+2k+1}}{\varepsilon\cdot s_{n+2k}}\right)}_{=P_{n+2k}}.\label{anflem3}
\end{align}
Lemma \ref{lemqnk} tells us that
\begin{align*}
    Q_{n,k} = \frac{s_{n+2k}}{\varepsilon^{2k}\cdot s_n}
    &< 1-\frac{2k}{2n}+\frac{\frac{3}{8}(2k)^2+(\frac{19}{72}-S)\cdot 2k}{n^2}+\frac{1.2(2k)^6}{n^3}\\
    &=1-\frac{k}{n}+\frac{\frac{3}{2}k^2+(\frac{19}{36}-2S)k}{n^2}+\frac{76.8 k^6}{n^3}.
\end{align*}
From Lemma \ref{lemOhne1} we deduce
\begin{align*}
    P_{n+2k} &= 1-\varepsilon\cdot\frac{s_{n+2k+1}}{\varepsilon\cdot s_{n+2k}}\\
    &< 1- \varepsilon\cdot\left(1 -\frac{1}{2(n+2k+1)}+\frac{\frac{5}{36}-S}{(n+2k+1)^2}+\frac{0.0576}{(n+2k+1)^3}\right)\\
    &= 1- \varepsilon+\frac{\varepsilon}{2(n+2k+1)}-\frac{(\frac{5}{36}-S)\varepsilon}{(n+2k+1)^2}-\frac{0.0576\varepsilon}{(n+2k+1)^3}.
\end{align*}
Here we write $K=2k+1$ and estimate $\frac{1}{n+K}<\frac{1}{n}-\frac{K}{n^2}+\frac{K^2}{n^3}$, as well as $\frac{1}{(n+K)^2}>\frac{1}{n^2}-\frac{2K}{n^3}$ and $\frac{1}{(n+K)^3}>\frac{1}{n^3}$:
\begingroup
\allowdisplaybreaks
\begin{align*}
    P_{n+2k} &< 1- \varepsilon+\frac{\varepsilon}{2n}-\frac{\varepsilon K}{2n^2}+\frac{\varepsilon K^2}{2n^3}-\frac{(\frac{5}{36}-S)\varepsilon}{n^2}+\frac{2(\frac{5}{36}-S)\varepsilon K}{n^3}-\frac{0.0576\varepsilon}{n^3}\\
    &=1- \varepsilon+\frac{\varepsilon}{2n}
    -\frac{\frac{\varepsilon K}{2}+(\frac{5}{36}-S)\varepsilon}{n^2}
    +\frac{\frac{\varepsilon}{2} K^2+2(\frac{5}{36}-S)\varepsilon K-0.0576\varepsilon}{n^3}\\
    &<1- \varepsilon+\frac{\varepsilon}{2n}
    -\frac{(k+\frac{23}{36}-S)\varepsilon}{n^2}
    +\frac{(2k^2+(\frac{23}{9}-4S)k +0.67032)\varepsilon}{n^3}\\
    &\leq 1- \varepsilon+\frac{\varepsilon}{2n} -\frac{(k+\frac{23}{36}-S)\varepsilon}{n^2} +\frac{(4.456k^2 +0.67032)\varepsilon}{n^3}.
\end{align*}
\endgroup
Since both estimates are positive, we may multiply them and obtain
\begin{align*}
    Q_{n,k}\cdot P_{n+2k} &< \left(1-\frac{k}{n}+\frac{\frac{3}{2}k^2+(\frac{19}{36}-2S)k}{n^2}+\frac{76.8 k^6}{n^3}\right)\\
    &~~~~\cdot\left(1- \varepsilon+\frac{\varepsilon}{2n} -\frac{(k+\frac{23}{36}-S)\varepsilon}{n^2} +\frac{(4.456k^2 +0.67032)\varepsilon}{n^3}\right)\\
    &= 1- \varepsilon + \frac{\frac{\varepsilon}{2}-(1-\varepsilon)k}{n}
    + \frac{\alpha_2(k)}{n^2} + R_3(n,k)\,,\\
\text{with}\qquad
    \alpha_2(k) &= \left(\frac{3}{2}k^2+\left(\frac{19}{36}-2S\right)k\right)\cdot(1-\varepsilon)-k\cdot\frac{\varepsilon}{2}-\left(k+\frac{23}{36}-S\right)\varepsilon\\
    &= \frac{3}{2}(1-\varepsilon)k^2
    +\left[\left(\frac{19}{36} - 2 S\right)(1 -\varepsilon)-\frac{3}{2}\varepsilon\right]k
    -\left(\frac{23}{36}-S\right)\varepsilon\,.
\end{align*}
To estimate the remaining term $R_3(n,k)$, we first omit all negative terms, then we use $\frac{1}{n^l}\leq\frac{1}{n^3}$ for $l\geq 3$ and $k^l\leq k^8$ for $1\leq l \leq 8$. Since $\varepsilon<10^{-14}$, this yields 
\begin{align*}
    R_3(n,k) &<\frac{76.9k^8 +0.67032\varepsilon}{n^3}.
\end{align*}
From (\ref{anflem3}) we obtain
\begin{align*}
    \frac{\left|\sum_{k=n}^\infty (-1)^k s_k\right| }{s_n}
     &=\sum_{k=0}^\infty \varepsilon^{2k}\cdot Q_{n,k}\cdot P_{n+2k}\\
     &<\sum_{k=0}^\infty \varepsilon^{2k}\left(1- \varepsilon + \frac{\frac{\varepsilon}{2}-(1-\varepsilon)k}{n}
    + \frac{\alpha_2(k)}{n^2} + \frac{76.9k^8 +0.67032\varepsilon}{n^3}\right)\\
    &=:c_0 + \frac{c_1}{n} + \frac{c_2}{n^2}+\frac{c_3}{n^3}.
\end{align*}
The numbers $c_0$, $c_1$ and $c_2$ can be computed using the geometric series and its moments
$\sum_{k=0}^\infty k\,\varepsilon^{2k}= \frac{\varepsilon^2}{(1-\varepsilon^2)^2}$ and
$\sum_{k=0}^\infty k^2\,\varepsilon^{2k} = \frac{\varepsilon^2(1+\varepsilon^2)}{(1-\varepsilon^2)^3}$:
{\allowdisplaybreaks
\begin{align*}
     c_0 &= (1-\varepsilon)\cdot \frac{1}{1-\varepsilon^2} = \frac{1}{1+\varepsilon},\\
     c_1 &= \frac{\varepsilon}{2}\cdot\frac{1}{1-\varepsilon^2} -(1-\varepsilon)\cdot \frac{\varepsilon^2}{(1-\varepsilon^2)^2} = \frac{\varepsilon}{2(1+\varepsilon)^2},\\
     c_2 &= \frac{3}{2}(1-\varepsilon)\cdot \frac{\varepsilon^2(1+\varepsilon^2)}{(1-\varepsilon^2)^3}
    +\left[\left(\frac{19}{36} - 2 S\right)(1 -\varepsilon)-\frac{3}{2}\varepsilon\right]\cdot \frac{\varepsilon^2}{(1-\varepsilon^2)^2}
    - \frac{\left(\frac{23}{36}-S\right)\varepsilon}{1-\varepsilon^2}\\
    &= \frac{\varepsilon S}{(1+\varepsilon)^2} + \frac{4\varepsilon^2 - 23\varepsilon}{36 (1+\varepsilon)^3}.
\end{align*}}
Here we recognize the values of $a_1=(1+\varepsilon)c_1$ and $a_2=(1+\varepsilon)c_2$ from Theorem~\ref{theo01}.
To estimate $c_3$, we use $\varepsilon^{k}\cdot k^8\leq \varepsilon$ (valid for $k\geq 1$) and thus
\begin{align*}
\sum_{k=0}^\infty \varepsilon^{2k}\cdot k^8 =\sum_{k=1}^\infty \varepsilon^{2k}\cdot k^8
    < \sum_{k=1}^\infty \varepsilon^{k}\cdot \varepsilon=\frac{\varepsilon^2}{1-\varepsilon},
\end{align*}
which implies that
\begin{align*}
   c_3 < 76.9\cdot\frac{\varepsilon^2}{1-\varepsilon}  +0.67032\varepsilon\cdot\frac{1}{1-\varepsilon^2}<\frac{0.6704\varepsilon}{1+\varepsilon}.
\end{align*}
This proves the upper bound, $e_n<0.6704$.
\qed

\noindent\textbf{Proof of lower bound in Theorem \ref{theo01}}:
 Since $s_n$ decreases monotonically to zero, we obtain
\begin{align*}
    \left|\sum_{k={n-1}}^\infty (-1)^k s_k\right| 
    &= s_{n-1}-\left|\sum_{k=n}^\infty (-1)^k s_k\right|,
\end{align*}
which implies that
\begin{subequations}
\begin{align}
    \underbrace{\frac{1+\varepsilon}{s_{n-1}}\cdot\left|\sum_{k=n-1}^\infty (-1)^k s_k\right| }_{=f(n-1)}
    &= 1+\varepsilon-\underbrace{\frac{1+\varepsilon}{s_{n}}\cdot\left|\sum_{k=n}^\infty (-1)^k s_k\right|}_{=f(n)}\cdot\,\frac{s_n}{s_{n-1}}.\label{eqdefifn}
\end{align}
Here we observe a relation between $f(n-1)$ and $-f(n)$, i.e.
\begin{align}
    f(n-1) = 1 + \varepsilon -  f(n)\cdot\frac{s_n}{\varepsilon\cdot s_{n-1}}\cdot \varepsilon,\label{eqdefifn2}
\end{align}
\end{subequations}
which allows us to transform the already proven upper bound $e_n<0.6704$ from Theorem \ref{theo01} into a lower bound:
Using $e_n<0.6704$ and Lemma \ref{lemOhne1} yields for $n\geq 2$
{\allowdisplaybreaks\begin{align}
    f(n)\cdot\frac{s_n}{\varepsilon\cdot s_{n-1}} &< \left(1+\frac{\varepsilon}{2(1+\varepsilon)n}
+\frac{\frac{\varepsilon S}{(1+\varepsilon)} + \frac{4\varepsilon^2 - 23\varepsilon}{36 (1+\varepsilon)^2}}{n^2}+\frac{0.6704\varepsilon}{n^3}\right)\nonumber\\
&~~~~~\cdot\left(1 -\frac{1}{2n}+\frac{\frac{5}{36}-S}{n^2}+\frac{0.0576}{n^3} + \frac{0.11}{n^4}\right)\nonumber\\
&< 1 -\frac{1}{2(1+\varepsilon)n}+\frac{C}{n^2}+\frac{0.113}{n^3},\nonumber\\
\text{with}\qquad
C &= \left(\frac{\varepsilon S}{1+\varepsilon} + \frac{4\varepsilon^2 - 23\varepsilon}{36 (1+\varepsilon)^2}\right)+\left(\frac{5}{36}-S\right)-\frac{\varepsilon}{4(1+\varepsilon)}\nonumber\\
&=\frac{5-22\varepsilon}{36(1+\varepsilon)^2}-\frac{S}{1+\varepsilon}.\label{eqC}
\end{align}}
This proves the lower bound on $f(n-1)$ for $n\geq 2$:
\begin{align*}
    f(n-1) &= 1+ \varepsilon - \varepsilon\cdot f(n)\cdot\frac{s_n}{\varepsilon\cdot s_{n-1}}\\
    &> 1+\frac{\varepsilon}{2(1+\varepsilon)n}-\frac{C\varepsilon}{n^2}-\frac{0.113\varepsilon}{n^3}.
\end{align*}
For $n\geq 2$ we have that $\frac{1}{(n+1)^2}<\frac{1}{n^2}-\frac{1}{n^3}$. This shows that for $n\geq 2$ we have that
\begin{align*}    f(n) &>1+\frac{\varepsilon}{2(1+\varepsilon)(n+1)}-\frac{C\varepsilon}{(n+1)^2}-\frac{0.113\varepsilon}{(n+1)^3}\\
    &>1+\frac{\varepsilon}{2(1+\varepsilon)n}-\frac{\varepsilon}{2(1+\varepsilon)n^2}
    +\frac{\varepsilon}{2(1+\varepsilon)n^3(1+1/n)}
    -\frac{C\varepsilon}{n^2}+\frac{C\varepsilon}{n^3}
    -\frac{0.113\varepsilon}{n^3}\\
    &>1+\frac{\varepsilon}{2(1+\varepsilon)n}-\frac{\frac{\varepsilon}{2(1+\varepsilon)}+C\varepsilon}{n^2}
    +\frac{0.334\varepsilon}{n^3}
    =1+\frac{a_1}{n}+\frac{a_2}{n^2} +\frac{0.334\varepsilon}{n^3}.
\end{align*}
Using the value of $C$ from (\ref{eqC}), we recognize $a_1$ and $a_2$ as in Theorem \ref{theo01} and we have proven $e_n>0.334$ for $n\geq 2$.
In the case $n=1$ we have:
\begin{align*}
    \frac{\sqrt{640320^3}}{12\pi\cdot \numprint{545140134}} &= \sum_{k=0}^\infty (-1)^k s_k= s_0 - \left|\sum_{k=1}^\infty (-1)^k s_k\right|,
\end{align*}
which implies that
\begin{align*}\ln\left(\left|\sum_{k=1}^\infty (-1)^k s_k\right|\cdot\frac{1+\varepsilon}{s_1}\right)
   &=\ln\left(\frac{(S - \frac{\sqrt{640320^3}}{12\pi\cdot \numprint{545140134}})(1+\varepsilon)}{\frac{720}{6\cdot 1^3}\cdot\frac{1+S}{640320^3}}\right) > 0.20767\varepsilon\,.
\end{align*}
This yields $a_1 + a_2 +e_1\varepsilon> 0.20767\varepsilon$ and $e_1 > 0.3216$, which finishes the proof of the lower bound on $e_n$ and of Theorem \ref{theo01}.
\qed

  \section{Sharper estimates in higher orders}\label{secrekurs}
  \renewcommand{\leftmark}{Sharper estimates in higher orders}
    Using the terms $s_n$ of the Chudnovskys' series from Definition \ref{defi1}, we denote the truncated part of the series by $f(n)$ as in (\ref{eqdefifn}). This can be written as a formal power series
\begin{align}
    f(n)=\frac{1+\varepsilon}{s_n}\cdot\left|\sum_{k=n}^\infty (-1)^k s_k\right|=:\sum_{k=0}^\infty \frac{a_k}{n^k},\label{defipoweran}
\end{align}
which we use to define the coefficients $a_k$. For $k\leq 2$, this definition is consistent with the values of $a_{k}$ from Theorem \ref{theo01}.
In this section, we explain a method how to compute the rational coefficients $a_k$ efficiently to higher orders (see (\ref{eqrekurs})), but without error bounds like (\ref{th1errorbound}).

In Lemma \ref{lemOhne1} we proved
\begin{align}
  \frac{s_{n}}{\varepsilon\cdot s_{n-1}} 
    =1 - \frac{1}{2 n} &+ \frac{\frac{5}{36} - S}{n^2} + \frac{ \frac{5}{72}- \frac{S}{2}+S^2}{n^3} \label{defibk}\\
    &    + \sum_{k=4}^\infty\frac{\left(\frac{5}{72}-\frac{23S}{36}+\frac{3S^2}{2}-S^3\right)\left(1-S\right)^{k-4}}{n^k} =: \sum_{k=0}^\infty \frac{b_k}{n^k} ,\nonumber
\end{align}
which gives $\frac{s_{n}}{s_{n-1}}=\varepsilon\cdot\sum_{k=0}^\infty \frac{b_k}{n^k}$ as a power series. Thus (\ref{eqdefifn2}) reads
\begin{align*}
f(n-1) &= 1 + \varepsilon -  f(n)\cdot\frac{s_n}{s_{n-1}}\,,\\
    \text{which yields}\qquad\sum_{k=0}^\infty \frac{a_k}{(n-1)^k} &= 1+\varepsilon - \sum_{k=0}^\infty \frac{a_k}{n^k}\cdot\varepsilon\cdot\sum_{k=0}^\infty \frac{b_k}{n^k}.
\end{align*}
Using $\frac{1}{\left(1-x\right)^k}
    = \sum_{l=0}^\infty\binom{l+k-1}{k-1}\cdot x^l$ with $x=\frac{1}{n}$ yields
\begin{align*}
    a_0 + \sum_{k=1}^\infty\sum_{l=0}^\infty \frac{a_k\cdot\binom{l+k-1}{k-1}}{n^{k+l}} &= 1+\varepsilon - \varepsilon\sum_{k=0}^\infty \sum_{l=0}^\infty \frac{a_k\cdot b_l}{n^{k+l}}.
\end{align*}
Here, $b_0=1$ yields $a_0 = 1$ and equating the coefficients for $N=k+l\geq 1$ gives the recursion formula:
\begin{align}
     -(1+\varepsilon)\cdot a_N = \varepsilon \cdot b_N+\sum_{k=1}^{N-1} a_k\cdot \left(\binom{N-1}{k-1}+\varepsilon\cdot b_{N-k}\right).\label{eqrekurs}
\end{align}
Using recursion (\ref{eqrekurs}) and the values of $b_k$ from (\ref{defibk}), we compute further terms of the sequence $a_k$:
{\allowdisplaybreaks\begin{align*}
a_{1} &= \tnfrac{1}{303862746112002} \approx 3.2910\cdot 10^{-15}, \\ 
a_{2} &= \tnfrac{-62186213362465}{15388761412454497761254741334} \approx -4.0410\cdot 10^{-15}, \\
a_{3} &= \tnfrac{20630598257083699331942595295}{4676071302050834345503446765524996702890668} \approx 4.4120\cdot 10^{-15}, \\
a_{4} &= \tnfrac{-1933230018398723806508418321549998750727169}{405966819101911798225167895714948764118181484848996742096}\\
&\approx -4.7620\cdot 10^{-15} \\
a_{10} &\approx -6.6157\cdot 10^{-15},\\
a_{100} &\approx -1.4296\cdot 10^{5},\\
a_{1000} &\approx -1.3214\cdot 10^{1049}   .
\end{align*}}
Figure \ref{fig_a200} shows the values of $|a_k|$ for $k\leq 200$.
We observe numerically that the $a_k$ have alternating signs for most $k$ and that they grow roughly like $|a_k|\approx k!/q^k$ with $q=\ln(53360^3)$, thus (\ref{defipoweran}) diverges for all $n\in\mathbb N$.

\begin{figure}[!ht]\centering\begin{tikzpicture}\begin{semilogyaxis}[ytick={1e-15,1,1e15,1e30,1e45,1e60,1e75}]
\addplot[only marks,mark=o,mark size = 1pt]  coordinates {
(0, 1.0)
(47, 2.0082355804684328e-15)
(80, 1.879822258077603e-05)
(113, 54579149491.16738)
(146, 2.1140026418083852e+30)
(179, 1.2949226332250464e+53)
};
\addplot[only marks,mark=*,mark size = 0.25pt]  coordinates { 
(0, 1.0)
(1, 3.2909595295745995e-15)
(2, 4.041014848156408e-15)
(3, 4.411951171068653e-15)
(4, 4.762039475727241e-15)
(5, 5.1002121578284976e-15)
(6, 5.426691919455015e-15)
(7, 5.741484313004556e-15)
(8, 6.044589476908377e-15)
(9, 6.336007414616045e-15)
(10, 6.615738126209871e-15)
(11, 6.8837816116843195e-15)
(12, 7.140137871023677e-15)
(13, 7.384806904195618e-15)
(14, 7.617788711133947e-15)
(15, 7.839083291703229e-15)
(16, 8.048690645626639e-15)
(17, 8.246610772338888e-15)
(18, 8.432843670686586e-15)
(19, 8.607389338318345e-15)
(20, 8.770247770444168e-15)
(21, 8.921418957313615e-15)
(22, 9.060902879092885e-15)
(23, 9.188699495464522e-15)
(24, 9.304808724525828e-15)
(25, 9.409230399999728e-15)
(26, 9.501964184516474e-15)
(27, 9.583009393960886e-15)
(28, 9.652364641862293e-15)
(29, 9.710027119820795e-15)
(30, 9.755991142163656e-15)
(31, 9.790245203887874e-15)
(32, 9.812766035849041e-15)
(33, 9.82350659783493e-15)
(34, 9.822371838281023e-15)
(35, 9.80916977719186e-15)
(36, 9.783512831671371e-15)
(37, 9.744618851773736e-15)
(38, 9.690910092616495e-15)
(39, 9.61920522109651e-15)
(40, 9.523091977648919e-15)
(41, 9.389650843156529e-15)
(42, 9.192861153153668e-15)
(43, 8.880335074592884e-15)
(44, 8.346637432437416e-15)
(45, 7.379645865160209e-15)
(46, 5.5527456633778915e-15)
(47, 2.0082355804684328e-15)
(48, 4.977694026520427e-15)
(49, 1.8868403077456763e-14)
(50, 4.661994337627207e-14)
(51, 1.0219877789672568e-13)
(52, 2.1363792422107258e-13)
(53, 4.371862923601025e-13)
(54, 8.856726715798324e-13)
(55, 1.785346524528658e-12)
(56, 3.589748947766366e-12)
(57, 7.207758795589076e-12)
(58, 1.4460138207761355e-11)
(59, 2.8993201765461314e-11)
(60, 5.810659921688739e-11)
(61, 1.1640818087664724e-10)
(62, 2.331186526151132e-10)
(63, 4.666618830775e-10)
(64, 9.337852941054525e-10)
(65, 1.8676221587254982e-09)
(66, 3.733302226469138e-09)
(67, 7.457707808586114e-09)
(68, 1.4884810353367868e-08)
(69, 2.9674386064815403e-08)
(70, 5.9064962322680835e-08)
(71, 1.1729950268286698e-07)
(72, 2.3218351044441445e-07)
(73, 4.573376563444844e-07)
(74, 8.941471190546127e-07)
(75, 1.7280651689002767e-06)
(76, 3.278581391638009e-06)
(77, 6.031522339120319e-06)
(78, 1.0501246165725048e-05)
(79, 1.6348586747537364e-05)
(80, 1.879822258077603e-05)
(81, 3.982720808738019e-06)
(82, 0.00013250336109887616)
(83, 0.0006383465101880013)
(84, 0.002396554302883627)
(85, 0.008153505980392776)
(86, 0.026392916584382645)
(87, 0.08306087586476485)
(88, 0.2570000125266866)
(89, 0.7867772456918646)
(90, 2.3922099526656915)
(91, 7.24089914005621)
(92, 21.850941058125553)
(93, 65.80137974980528)
(94, 197.85074491557998)
(95, 594.1927314168967)
(96, 1782.7158869143498)
(97, 5343.5165744241895)
(98, 16000.863967481342)
(99, 47859.75616721114)
(100, 142956.3201862908)
(101, 426269.3502974633)
(102, 1268196.8679096587)
(103, 3761777.860005924)
(104, 11113659.56823421)
(105, 32654773.973062713)
(106, 95225604.00699961)
(107, 274752485.7319567)
(108, 780684043.8842802)
(109, 2168213797.4202166)
(110, 5810864631.729396)
(111, 14662723478.690226)
(112, 32925545622.269943)
(113, 54579149491.16738)
(114, 12888035366.693161)
(115, 744654557987.1847)
(116, 5056286653705.428)
(117, 26452521492385.773)
(118, 124470627776728.36)
(119, 553770591379174.06)
(120, 2382298919262753.5)
(121, 1.002857289581635e+16)
(122, 4.160058998572747e+16)
(123, 1.7079927355717974e+17)
(124, 6.960642970634415e+17)
(125, 2.8211815541448187e+18)
(126, 1.1386869154953361e+19)
(127, 4.580980201498246e+19)
(128, 1.8380209512239838e+20)
(129, 7.357637160699648e+20)
(130, 2.939009742250695e+21)
(131, 1.1715202150964033e+22)
(132, 4.65943481701096e+22)
(133, 1.848570839761916e+23)
(134, 7.312605110250279e+23)
(135, 2.882504964301818e+24)
(136, 1.1312284029276437e+25)
(137, 4.414563874798357e+25)
(138, 1.7102258393244854e+26)
(139, 6.561735029976127e+26)
(140, 2.4848116663749835e+27)
(141, 9.239198978859648e+27)
(142, 3.345590637662896e+28)
(143, 1.1632336843167246e+29)
(144, 3.7783533492400305e+29)
(145, 1.0743363719284075e+30)
(146, 2.1140026418083852e+30)
(147, 2.4511599848184068e+30)
(148, 6.428659839837536e+31)
(149, 5.29249124610563e+32)
(150, 3.475737115629003e+33)
(151, 2.068627823264712e+34)
(152, 1.1660041412185884e+35)
(153, 6.353124897456314e+35)
(154, 3.3836131134054264e+36)
(155, 1.7733099550013996e+37)
(156, 9.184635980629339e+37)
(157, 4.714777469874982e+38)
(158, 2.4034964226088942e+39)
(159, 1.2184595847614617e+40)
(160, 6.148722456070331e+40)
(161, 3.0906636415852515e+41)
(162, 1.548085577471233e+42)
(163, 7.728810637108602e+42)
(164, 3.8461676731425207e+43)
(165, 1.907630544284562e+44)
(166, 9.42738238420612e+44)
(167, 4.640011345777511e+45)
(168, 2.272915753187984e+46)
(169, 1.1070558424644796e+47)
(170, 5.35431144579662e+47)
(171, 2.5667792875823907e+48)
(172, 1.2164637763113847e+49)
(173, 5.678087992428711e+49)
(174, 2.5955190548388472e+50)
(175, 1.1513065403157659e+51)
(176, 4.876845148433477e+51)
(177, 1.9104822348686964e+52)
(178, 6.386036175659199e+52)
(179, 1.2949226332250464e+53)
(180, 4.898593279812956e+53)
(181, 9.261213745940633e+54)
(182, 8.708876624379776e+55)
(183, 6.79500382044363e+56)
(184, 4.8572956708937355e+57)
(185, 3.301316107503358e+58)
(186, 2.1720030758745345e+59)
(187, 1.3972250253800478e+60)
(188, 8.842361327790931e+60)
(189, 5.527139059294747e+61)
(190, 3.421690919042422e+62)
(191, 2.1018978526755253e+63)
(192, 1.282904518930317e+64)
(193, 7.787517634965472e+64)
(194, 4.704435666753465e+65)
(195, 2.829439013885855e+66)
(196, 1.6946231839390953e+67)
(197, 1.0107544377781954e+68)
(198, 6.002945013981984e+68)
(199, 3.5489482265877906e+69)
(200, 2.087549816316693e+70)
};
\addplot[solid,color=darkgray] coordinates { 
(1, 0.030623697496693213)
(2, 0.0018756216967379482)
(3, 0.00017231541437741216)
(4, 2.1107740495644843e-05)
(5, 3.2319852988866444e-06)
(6, 5.938520406411852e-07)
(7, 1.2730161675272734e-07)
(8, 3.118756961820396e-08)
(9, 8.59570827880444e-09)
(10, 2.632323700999287e-09)
(11, 8.867263320605589e-10)
(12, 3.258580674644987e-10)
(13, 1.2972672550356813e-10)
(14, 5.561796798680958e-11)
(15, 2.5548417405132354e-11)
(16, 1.25181920981444e-11)
(17, 6.51700657632418e-12)
(18, 3.5923470835916157e-12)
(19, 2.0902080574297168e-12)
(20, 1.280197985117568e-12)
(21, 8.232923124744477e-13)
(22, 5.546696040285513e-13)
(23, 3.906787857807621e-13)
(24, 2.8713669489901126e-13)
(25, 2.1982968211969035e-13)
(26, 1.7503193983671828e-13)
(27, 1.447233798011148e-13)
(28, 1.2409502010439466e-13)
(29, 1.1020720233916829e-13)
(30, 1.0124856079174651e-13)
(31, 9.611876422752194e-14)
(32, 9.419238270270749e-14)
(33, 9.518912813455559e-14)
(34, 9.911146420886437e-14)
(35, 1.0623058244353096e-13)
(36, 1.1711423597933597e-13)
(37, 1.3269942460193857e-13)
(38, 1.544223874058106e-13)
(39, 1.8443039466667576e-13)
(40, 2.259176246187208e-13)
(41, 2.836557528153031e-13)
(42, 3.648366946313303e-13)
(43, 4.804238885995817e-13)
(44, 6.473436567249797e-13)
(45, 8.920825343977047e-13)
(46, 1.2566678210720215e-12)
(47, 1.8087393146978944e-12)
(48, 2.6587337099368565e-12)
(49, 3.989592586010991e-12)
(50, 6.108803824452529e-12)
(51, 9.540782179720518e-12)
(52, 1.5193049422387327e-11)
(53, 2.465916952686545e-11)
(54, 4.0778367197961764e-11)
(55, 6.868314098137021e-11)
(56, 1.1778657702202797e-10)
(57, 2.056024487219318e-10)
(58, 3.651858172659094e-10)
(59, 6.598170598839998e-10)
(60, 1.2123622827027071e-09)
(61, 2.264747963915178e-09)
(62, 4.300007306297798e-09)
(63, 8.295973747842967e-09)
(64, 1.625941699164488e-08)
(65, 3.236502538260722e-08)
(66, 6.541502528812423e-08)
(67, 1.3421774639285756e-07)
(68, 2.794965691671474e-07)
(69, 5.905860686021436e-07)
(70, 1.2660150377443347e-06)
(71, 2.752674369492447e-06)
(72, 6.069388838273095e-06)
(73, 1.356830032743495e-05)
(74, 3.07478528331022e-05)
(75, 7.062097078753224e-05)
(76, 0.00016436331872033487)
(77, 0.0003875727665073963)
(78, 0.0009257750704393631)
(79, 0.0022397018008620774)
(80, 0.005487036034591943)
(81, 0.013610699865820699)
(82, 0.034178416343552426)
(83, 0.08687356709074175)
(84, 0.22347274647982687)
(85, 0.5817027518910416)
(86, 1.531994463193851)
(87, 4.081634145648793)
(88, 10.999536184429978)
(89, 29.97933571571641)
(90, 82.6270297298919)
(91, 230.26140987836226)
(92, 648.7339300374888)
(93, 1847.5967415257367)
(94, 5318.542908585862)
(95, 15472.97766979305)
(96, 45488.61960315181)
(97, 135123.88344809978)
(98, 405523.3072667178)
(99, 1229443.685869859)
(100, 3765011.1525298078)
(101, 11645154.82328166)
(102, 36375005.258254215)
(103, 114735527.2193484)
(104, 365417112.07964927)
(105, 1174994425.546454)
(106, 3814163427.912955)
(107, 12498005211.071173)
(108, 41335394136.76141)
(109, 137976844048.6337)
(110, 464789724706.3151)
(111, 1579927372116.5034)
(112, 5418920405969.42)
(113, 18752053857630.402)
(114, 65465323624666.055)
(115, 230550880728563.12)
(116, 818997169767214.9)
(117, 2934444424579093.0)
(118, 1.060389752874184e+16)
(119, 3.8642935474528616e+16)
(120, 1.4200674796274392e+17)
(121, 5.262013747441272e+17)
(122, 1.9659362701460073e+18)
(123, 7.405121229071627e+18)
(124, 2.8119751863233393e+19)
(125, 1.0764134684271685e+20)
(126, 4.153433815248472e+20)
(127, 1.6153574592799609e+21)
(128, 6.331931927298105e+21)
(129, 2.5014024660548774e+22)
(130, 9.958285016933106e+22)
(131, 3.994969554071796e+23)
(132, 1.6148977565476015e+24)
(133, 6.577400671152102e+24)
(134, 2.699086001469979e+25)
(135, 1.1158559085587743e+26)
(136, 4.6473421959298267e+26)
(137, 1.9497675815332899e+27)
(138, 8.239854779591373e+27)
(139, 3.507454000597742e+28)
(140, 1.5037569441702033e+29)
(141, 6.493134285121634e+29)
(142, 2.823581678172206e+30)
(143, 1.2364997097245795e+31)
(144, 5.452731801411046e+31)
(145, 2.4212507336466584e+32)
(146, 1.0825556904502899e+33)
(147, 4.8733231241880095e+33)
(148, 2.2087397627498477e+34)
(149, 1.007832697316202e+35)
(150, 4.629534547483161e+35)
(151, 2.1407793295424837e+36)
(152, 9.964903946452575e+36)
(153, 4.668981721808451e+37)
(154, 2.201914851540278e+38)
(155, 1.0451770020692072e+39)
(156, 4.993120757331423e+39)
(157, 2.400652768300542e+40)
(158, 1.1615664539023522e+41)
(159, 5.655862093351351e+41)
(160, 2.7712545572768928e+42)
(161, 1.3663435860989186e+43)
(162, 6.778483810472925e+43)
(163, 3.38359047447993e+44)
(164, 1.6993360387479228e+45)
(165, 8.586592211314479e+45)
(166, 4.36502315995216e+46)
(167, 2.2323415852344623e+47)
(168, 1.1484908973805473e+48)
(169, 5.943905391425807e+48)
(170, 3.0944141311517917e+49)
(171, 1.6204370790202374e+50)
(172, 8.535289286298636e+50)
(173, 4.521910626701504e+51)
(174, 2.4095106426218936e+52)
(175, 1.2912921881075254e+53)
(176, 6.959768877326835e+53)
(177, 3.772469264406981e+54)
(178, 2.0563798447241444e+55)
(179, 1.127233782026012e+56)
(180, 6.213611942787241e+56)
(181, 3.444136282209396e+57)
(182, 1.9195938151165438e+58)
(183, 1.0757666036849365e+59)
(184, 6.0616869928858295e+59)
(185, 3.434178472610848e+60)
(186, 1.9561107141249642e+61)
(187, 1.1201925099748443e+62)
(188, 6.449234073943851e+62)
(189, 3.7327385346140525e+63)
(190, 2.1718948586471422e+64)
(191, 1.2703687168854168e+65)
(192, 7.469450360670971e+65)
(193, 4.414724234416849e+66)
(194, 2.622786482030381e+67)
(195, 1.5662286865702341e+68)
(196, 9.400887847601316e+68)
(197, 5.671431929212057e+69)
(198, 3.438868272310889e+70)
(199, 2.0956861478742714e+71)
(200, 1.2835531728102393e+72)
};
\end{semilogyaxis}\end{tikzpicture}
\caption{Values of $|a_k|$. Simple dot (\raisebox{0.75mm}{\begin{tikzpicture}\fill[black]circle(0.5pt);\end{tikzpicture}}): $a_k a_{k+1}<0$. Dot with circle (\raisebox{0.5mm}{\begin{tikzpicture}\fill[black]circle(1.2pt);\fill[white]circle(0.9pt);\fill[black]circle(0.5pt);\end{tikzpicture}}): $a_k a_{k+1}>0$.\\
Gray line: rough approximations $|a_k|\approx k!/q^k$ with $q=\ln(53360^3)$.}
\label{fig_a200}
\end{figure}
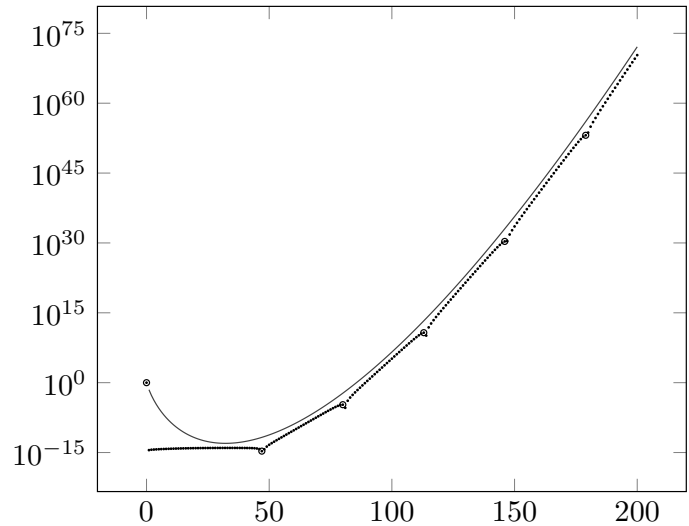

  \section{Proof of Theorem \ref{theo02}}\label{secttheo2}
  \renewcommand{\leftmark}{Proof of Theorem \ref{theo02}}
  From $\frac{1}{\pi}-\frac{1}{\pi_n}=\frac{\pi_n-\pi}{\pi~\pi_n}$ we deduce using Definition \ref{defi1}:
\begin{align}
    \left|\pi_n-\pi\right|&=\left|\pi~\pi_n\right|\cdot\left|\frac{1}{\pi}-\frac{1}{\pi_n}\right|
    =\pi\,\pi_n\,\frac{12\cdot \numprint{545140134}}{\sqrt{640320^3}}\,\left|\sum_{k=n}^\infty (-1)^k s_k\right|\label{exakte}\\
    &=\frac{12\cdot\numprint{545140134}\,\pi^2}{\sqrt{640320^3}(1+\varepsilon)}\cdot\frac{\pi_n}{\pi}\cdot s_n\cdot \underbrace{\frac{1+\varepsilon}{s_n}\left|\sum_{k=n}^\infty (-1)^k s_k\right|}_{\text{Theorem \ref{theo01}}}.\label{eq9}
\end{align}
Most of the work towards proving Theorem \ref{theo02} consists in establishing the error bound in Theorem \ref{theo01}. We still have to find approximations for $s_n$ (Lemma \ref{lem2b}) and for $\frac{\pi_n}{\pi}$ (Lemma \ref{lemma-pin}).

\begin{lemma}\label{lem2b} For all $n\geq 1$ there is $h_n$ with $-0.0024 < h_n < -0.0005$ and
    $$s_{n}=\frac{\varepsilon^n}{2\,\pi^{3/2}\sqrt{n}}\cdot\exp\left(\frac{S-\frac{19}{72}}{n}+\frac{-\frac{1}{2}S^2}{n^2}+\frac{\frac{131}{15552}+\frac{1}{3}S^3}{n^3}+\frac{h_n}{n^4}\right).$$
\end{lemma} 
\begin{proof}
We start with Stirling's formula
$n!=\sqrt{2\pi n}\cdot\left(\frac{n}{e}\right)^n\cdot e^{r_n}$
with the enveloping series $r_n=\sum_{j=1}^\infty \frac{B_{2j}}{2j(2j-1)n^{2j-1}}$ proved in \cite{sasvari} which tells us that:
$$\frac{1}{12n}-\frac{1}{360n^3}+\frac{1}{1260n^5}-\frac{1}{1680n^7}<r_n<\frac{1}{12n}-\frac{1}{360n^3}+\frac{1}{1260n^5}.$$
This yields $\frac{(6n)!}{(3n)!(n!)^3} = \frac{1728^n}{2(\pi n)^{3/2}}\cdot e^{ y_n}$
with the following bounds for $ y_n=r_{6n}-r_{3n}-3r_{n}$:
\begin{align*}
     y_n &<-\frac{19}{72n}+\frac{131}{15552n^3}-\frac{3337}{1399680n^5}+\frac{3281}{1837080n^7},\\
     y_n &>-\frac{19}{72n}+\frac{131}{15552n^3}-\frac{3337}{1399680n^5}-\frac{1}{470292480n^7}.
\end{align*}
To approximate $S+n=n\cdot\exp\ln(1+x)$ with $x=\frac{S}{n}>0$, we use
$$x-\frac{1}{2}x^2+\frac{1}{3}x^3-\frac{1}{4}x^4< \ln(1+x) < x-\frac{1}{2}x^2+\frac{1}{3}x^3.$$
Combining these two error bounds, we obtain
\begin{align*}
    s_n&=\frac{\left(6n\right)!}{\left(3n\right)!\left(n!\right)^3}\cdot\frac{S + n}{640320^{3n}}
    =\frac{1728^n}{2(\pi n)^{3/2}}\cdot e^{ y_n}\cdot \frac{ n}{640320^{3n}}\cdot\left(1+\frac{S}{n}\right)\\
    &=\frac{53360^{-3n}}{2\,\pi^{3/2}\sqrt{n}}\cdot\exp\left(-\frac{19}{72n}+\frac{131}{15552n^3}+\frac{S}{n}-\frac{\frac{1}{2}S^2}{n^2}+\frac{\frac{1}{3}S^3}{n^3}+\frac{h_n}{n^4}\right),
\end{align*}
where $\frac{h_n}{n^4}$ contains the terms of higher order
\begin{align*}
    -\frac{3337}{1399680n^5}-\frac{1}{470292480n^7}-\frac{\frac{1}{4}S^4}{n^4} < \frac{h_n}{n^4} &< -\frac{3337}{1399680n^5}+\frac{3281}{1837080n^7},
\end{align*}
which proves $-0.0024 < h_n < -0.0005$ and the lemma
\end{proof}


\begin{lemma}\label{lemma-pin}
For $n\geq 1$ we have the really weak estimate
$\pi_n=\pi\exp\left(\frac{f_n\varepsilon}{n^{10}}\right)$
with $-4< f_n < 4$.
\end{lemma}
\begin{proof}
Since $s_n$ decreases monotonically to zero, (\ref{exakte}) yields
\begin{align*}
    \left|\pi_n-\pi\right|
    &<\pi\,\pi_n\,\frac{12\cdot \numprint{545140134}}{\sqrt{640320^3}}\,s_n.
\end{align*}
Using $s_n<\frac{\varepsilon^n}{2\,\pi^{3/2}}$ (see Lemma \ref{lem2b}) and $\pi_n<3.2$ yields
\begin{align*}
    \frac{\left|\pi_n-\pi\right|}{\pi} < \frac{12\cdot \numprint{545140134}}{\sqrt{640320^3}}\,\pi_n\,s_n< 3.7\varepsilon^n\leq\frac{3.7\varepsilon}{n^{10}},
\end{align*}
and thus
$\pi\left(1-\frac{3.7\varepsilon}{n^{10}}\right)<\pi_n<\pi\left(1+\frac{3.7\varepsilon}{n^{10}}\right)$,
which proves the lemma.
\end{proof}

 \begin{lemma}\label{lem7}
For $n\geq 1$ there is $d_n$ with $0.321<d_n < 0.671$ and
\begin{align*}
    \left|\sum_{k=n}^\infty (-1)^k s_k\right|&=\frac{s_n}{1+\varepsilon}\cdot\exp\left(\frac{a_1}{n}+ \frac{a_2-\frac{1}{2}a_1^2}{n^2}+\frac{d_n\varepsilon}{n^3}\right),
\end{align*}
where $a_1$ and $a_2$ are the coefficients from Theorem \ref{theo01}.
\end{lemma}

\begin{proof}
Theorem \ref{theo01} implies that
$\left|\sum_{k=n}^\infty (-1)^k s_k\right|=\frac{s_n}{1+\varepsilon}\cdot\left(1 + x_n\right)$
with $x_n = \frac{a_1}{n}+\frac{a_2}{n^2}+\frac{e_n\varepsilon}{n^3}$.
Since $x_n>0$ for all $n\geq 1$, we can use
$x-\frac{1}{2}x^2<\ln(1+x)<x-\frac{1}{2}x^2+\frac{1}{3}x^3$.

Here, $e_n<0.6704$ yields the upper bound on $d_n$:
\begin{align*}
    \ln(1+x_n) &< \frac{a_1}{n}+\frac{a_2}{n^2}+\frac{e_n\varepsilon}{n^3}-\frac{1}{2}\left(\frac{a_1}{n}+\frac{a_2}{n^2}+\frac{e_n\varepsilon}{n^3}\right)^2 +\frac{1}{3}\left(\frac{a_1}{n}+\frac{a_2}{n^2}+\frac{e_n\varepsilon}{n^3}\right)^3\\
    &< \frac{a_1}{n}+\frac{a_2-\frac{1}{2}a_1^2}{n^2}+\frac{0.671\varepsilon}{n^3}.
\end{align*}
In the last step, we used that $|a_1|$, $|a_2|$ and $\varepsilon$ are smaller than $10^{-14}$.

For the lower bound on $d_n$, we use $e_n>0.3216$:
\begin{align*}
    \ln(1+x_n) &> \frac{a_1}{n}+\frac{a_2}{n^2}+\frac{e_n\varepsilon}{n^3}-\frac{1}{2}\left(\frac{a_1}{n}+\frac{a_2}{n^2}+\frac{e_n\varepsilon}{n^3}\right)^2\\
    &> \frac{a_1}{n}+\frac{a_2-\frac{1}{2}a_1^2}{n^2}+\frac{0.321\varepsilon}{n^3}.
\end{align*}
This proves $0.321<d_n<0.671$, which finishes the proof of Lemma \ref{lem7}.
\end{proof}

\noindent\textbf{Proof of Theorem \ref{theo02}:}
We use Lemma \ref{lem2b}, Lemma \ref{lemma-pin} and Lemma \ref{lem7} in equation (\ref{eq9}) and obtain
\begin{align*}
    \left|\pi_n-\pi\right|&=\frac{12\cdot\numprint{545140134}\,\pi^2}{\sqrt{640320^3}(1+\varepsilon)}\cdot\frac{\pi_n}{\pi}\cdot s_n\cdot \frac{1+\varepsilon}{s_n}\left|\sum_{k=n}^\infty (-1)^k s_k\right|\\
    &=\frac{12\cdot\numprint{545140134}\pi^2}{\sqrt{640320^3}(1+\varepsilon)}\cdot\frac{\varepsilon^n}{2\pi^{3/2}\sqrt{n}}\cdot\exp(R_n)\,.
\end{align*}
The terms in the exponent $R_n$ of this expression are
\begin{align*}
    R_n &=  \frac{a_1}{n}+ \frac{a_2-\frac{1}{2}a_1^2}{n^2}+\frac{d_n\varepsilon}{n^3} +\frac{S-\frac{19}{72}}{n}+\frac{-\frac{1}{2}S^2}{n^2}+\frac{\frac{131}{15552}+\frac{1}{3}S^3}{n^3}+\frac{h_n}{n^4}+ \frac{f_n\varepsilon}{n^{10}}\,.
\end{align*}
Here we sort the terms as stated in Theorem \ref{theo02}:
    \begin{align*}
    A_0 &=\frac{6\sqrt{\pi}\cdot \numprint{545140134}}{\sqrt{\numprint{640320}^3}(1+\varepsilon)} = \frac{\numprint{106720}\cdot\sqrt{\numprint{10005}\pi}}{\numprint{1672209}},\\
    A_1 &= a_1 + S-\frac{19}{72}  = -\nfrac{1781843197433}{7456754505816},\\
    A_2 &= a_2-\frac{1}{2}a_1^2 -\frac{1}{2}S^2  = -\nfrac{1080096011925710088395}{3475199235000451148614116},
    \end{align*}
and the error term
\begin{align*}
    \frac{\delta_n}{n^3} &= \frac{d_n\varepsilon}{n^3} +\frac{\nfrac{131}{15552}+\frac{1}{3}S^3}{n^3}+\frac{h_n}{n^4}+\frac{f_n\varepsilon}{n^{10}}. 
\end{align*}
This proves $0.007<\delta_n<\numprint{0.008429}$ for $n\geq 2$. Using
$\delta_1 = \ln\left(\left|\pi_1-\pi\right|\cdot\frac{53360^{3}}{A_0}\right)-A_1-A_2$
and $\pi_1 = \frac{\sqrt{\numprint{640320}^3}}{12\cdot \numprint{13591409}}$ yields $\numprint{0.006907}<\delta_1< \numprint{0.007}$.
\qed

  \section{Numerical verification and further results}\label{secconcl}
  \renewcommand{\leftmark}{Numerical verification and further results}
  For a numerical analysis, we computed the partial sums $\sum_{k=0}^{n-1}(-1)^k s_k$.
Using these, we calculated the numerical values of $e_n$ from Theorem~\ref{theo01} and of $\delta_n$ from Theorem~\ref{theo02}.
Figure~\ref{fig_both} shows
the results and that the bounds we have proven are close to optimal.

\begin{figure}[!ht]
\centering
\begin{tikzpicture}
\begin{axis}[xtick={1,20,40,60,80,100}, ymin=0.3,ymax=0.7,width=7cm]
\draw[dashed] (-20,0.6704)--(120,0.6704);
\draw[dashed] (-20,0.3216)--(120,0.3216);
\addplot[solid]  coordinates { 
(1, 0.321629129275556)
(2, 0.434920661979421)
(3, 0.492676952753371)
(4, 0.527682535502951)
(5, 0.551165354227805)
(6, 0.568009931440682)
(7, 0.580682167266523)
(8, 0.59056130147437)
(9, 0.59847904668123)
(10, 0.604966736404485)
(11, 0.610379693280822)
(12, 0.614964533901292)
(13, 0.618897784958226)
(14, 0.622309130553307)
(15, 0.625295982630069)
(16, 0.627932931493378)
(17, 0.630278059179629)
(18, 0.632377266915621)
(19, 0.634267309142313)
(20, 0.635977963721413)
(21, 0.637533612248908)
(22, 0.638954409430527)
(23, 0.640257161007248)
(24, 0.64145599159622)
(25, 0.642562858848264)
(26, 0.643587953657384)
(27, 0.644540014834648)
(28, 0.645426578840807)
(29, 0.646254179693655)
(30, 0.647028510274465)
(31, 0.647754553458184)
(32, 0.648436689454265)
(33, 0.649078784245438)
(34, 0.649684262896973)
(35, 0.650256170672339)
(36, 0.650797224257619)
(37, 0.651309854913269)
(38, 0.651796244999469)
(39, 0.652258359032573)
(40, 0.652697970204728)
(41, 0.653116683121494)
(42, 0.653515953372106)
(43, 0.653897104435433)
(44, 0.654261342335436)
(45, 0.654609768388036)
(46, 0.654943390323215)
(47, 0.655263132018956)
(48, 0.655569842045006)
(49, 0.655864301182885)
(50, 0.656147229062421)
(51, 0.656419290033598)
(52, 0.656681098374572)
(53, 0.656933222921784)
(54, 0.657176191195637)
(55, 0.657410493084692)
(56, 0.657636584142561)
(57, 0.657854888544147)
(58, 0.658065801741628)
(59, 0.658269692855157)
(60, 0.658466906828684)
(61, 0.658657766377395)
(62, 0.658842573749899)
(63, 0.65902161232538)
(64, 0.659195148063497)
(65, 0.65936343082261)
(66, 0.659526695560083)
(67, 0.659685163426797)
(68, 0.659839042766575)
(69, 0.659988530030038)
(70, 0.660133810611312)
(71, 0.660275059615074)
(72, 0.660412442560618)
(73, 0.660546116028893)
(74, 0.660676228257817)
(75, 0.660802919690652)
(76, 0.660926323481672)
(77, 0.661046565962985)
(78, 0.661163767075919)
(79, 0.66127804077009)
(80, 0.661389495372937)
(81, 0.661498233932234)
(82, 0.661604354533872)
(83, 0.661707950596955)
(84, 0.661809111148087)
(85, 0.661907921076539)
(86, 0.66200446137183)
(87, 0.662098809345123)
(88, 0.662191038835708)
(89, 0.662281220403735)
(90, 0.662369421510238)
(91, 0.662455706685443)
(92, 0.662540137686207)
(93, 0.662622773643436)
(94, 0.66270367120018)
(95, 0.662782884641117)
(96, 0.662860466014032)
(97, 0.662936465243853)
(98, 0.663010930239795)
(99, 0.663083906996071)
(100, 0.663155439686628) }; 
\end{axis}
\end{tikzpicture}\hfill\begin{tikzpicture}
\begin{axis}[xtick={1,20,40,60,80,100},ymin=0.00681275,ymax=0.00855816,width=7cm] 
\draw[dashed] (-20,0.008429)--(120,0.00843);
\draw[dashed] (-20,0.006907)--(120,0.00690);
\addplot[solid]  coordinates { 
(1, 0.0069077628516136125)
(2, 0.007917660997407405)
(3, 0.00818282574791953)
(4, 0.008285921839538198)
(5, 0.008335844940815865)
(6, 0.008363605419479295)
(7, 0.008380573954526472)
(8, 0.008391682691322963)
(9, 0.008399343174097679)
(10, 0.008404845154614723)
(11, 0.008408928209030438)
(12, 0.008412040726901113)
(13, 0.008414467229369157)
(14, 0.008416395240055956)
(15, 0.008417952385183886)
(16, 0.008419227948368822)
(17, 0.008420285896695337)
(18, 0.00842117302509141)
(19, 0.008421924201944788)
(20, 0.008422565837182063)
(21, 0.008423118228959997)
(22, 0.008423597184817366)
(23, 0.008424015162665312)
(24, 0.008424382087468279)
(25, 0.008424705944803346)
(26, 0.008424993218316607)
(27, 0.008425249216277502)
(28, 0.008425478318227651)
(29, 0.008425684163306926)
(30, 0.008425869795497668)
(31, 0.008426037776690463)
(32, 0.008426190275466596)
(33, 0.008426329137378515)
(34, 0.008426455941006333)
(35, 0.008426572042987116)
(36, 0.008426678614427736)
(37, 0.008426776670535041)
(38, 0.008426867094869527)
(39, 0.00842695065930906)
(40, 0.008427028040568317)
(41, 0.008427099833936628)
(42, 0.008427166564756795)
(43, 0.008427228698059655)
(44, 0.008427286646685359)
(45, 0.008427340778157004)
(46, 0.008427391420520942)
(47, 0.008427438867327544)
(48, 0.008427483381893995)
(49, 0.008427525200965005)
(50, 0.00842756453786667)
(51, 0.008427601585232147)
(52, 0.008427636517364303)
(53, 0.008427669492289524)
(54, 0.008427700653548035)
(55, 0.008427730131758605)
(56, 0.008427758045989593)
(57, 0.008427784504963221)
(58, 0.008427809608115884)
(59, 0.00842783344653382)
(60, 0.008427856103780605)
(61, 0.008427877656630547)
(62, 0.008427898175719994)
(63, 0.0084279177261269)
(64, 0.008427936367887549)
(65, 0.008427954156458074)
(66, 0.008427971143127455)
(67, 0.008427987375387696)
(68, 0.008428002897266208)
(69, 0.008428017749624737)
(70, 0.008428031970428618)
(71, 0.008428045594989707)
(72, 0.008428058656185846)
(73, 0.008428071184659457)
(74, 0.008428083208997477)
(75, 0.008428094755894625)
(76, 0.008428105850301739)
(77, 0.008428116515560706)
(78, 0.0084281267735274)
(79, 0.008428136644683754)
(80, 0.008428146148240138)
(81, 0.008428155302228922)
(82, 0.0084281641235901)
(83, 0.008428172628249762)
(84, 0.008428180831192048)
(85, 0.008428188746525216)
(86, 0.008428196387542384)
(87, 0.00842820376677739)
(88, 0.008428210896056254)
(89, 0.008428217786544604)
(90, 0.008428224448791443)
(91, 0.008428230892769574)
(92, 0.008428237127912947)
(93, 0.008428243163151236)
(94, 0.00842824900694184)
(95, 0.00842825466729954)
(96, 0.008428260151824013)
(97, 0.008428265467725368)
(98, 0.008428270621847861)
(99, 0.008428275620691957)
(100, 0.008428280470434838) }; 
\end{axis}
\end{tikzpicture}
\caption{Numerically computed values of $e_n$ (left) and $\delta_n$ (right) along with the proven bounds (dashed lines).}
\label{fig_both}
\end{figure}
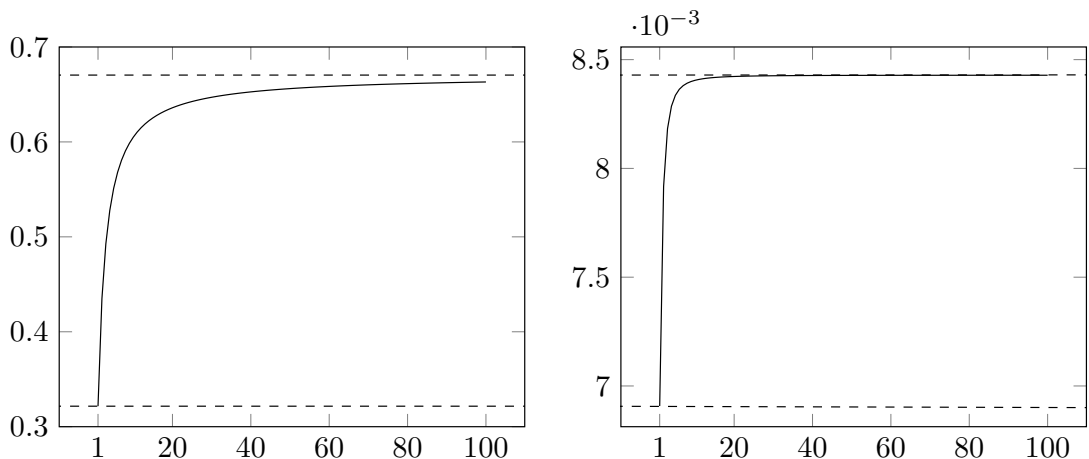

In \cite{cohenGuillera}, H. Cohen and J. Guillera list all 36 known rational hypergeometric series for $1/\pi$. They can be written with Pochhammer symbols $(a)_n={\Gamma(a+n)}/{\Gamma(n)}$ and $R\in\left\{1/2;1/3;1/4;1/6\right\}$:
\begin{align}
    \frac{1}{\pi}&= P\sum_{n=0}^\infty\frac{\left(\frac{1}{2}\right)_n \left(R\right)_n \left(1-R\right)_n}{n!^3}(n+S)Z^n.\label{typersz}
\end{align}

To generalize our results to formulae of this type, we can use the asymptotic estimate $\frac{(R)_n}{n!}\approx\frac{n^{R-1}}{\Gamma(R)}$ and Euler's reflection formula $\Gamma(R)\Gamma(1-R)={\pi}/{\sin(\pi R)}$ to obtain
\begin{align}\frac{\left(\frac{1}{2}\right)_n \left(R\right)_n \left(1-R\right)_n}{n!^3}\approx\frac{\sin(\pi R)}{(\pi n)^{3/2}}.\label{pochh}\end{align}

In Appendix \ref{appmorti}, we can use the method presented by C.~Mortici in \cite{morticiAMM} to prove a generalized version of Lemma \ref{lem2b}: For $n\geq 1$ and $0<R<1$ we have that
\begin{align}\exp\left(\frac{\sigma_1}{n}+\frac{\sigma_3}{n^3}+\frac{\sigma_5}{n^5}\right)
< \frac{\left(\frac{1}{2}\right)_n \left(R\right)_n \left(1-R\right)_n}{n!^3}\cdot\frac{(\pi n)^{3/2}}{\sin(\pi R)}
< \exp\left(\frac{\sigma_1}{n}+\frac{\sigma_3}{n^3}\right),\label{pochhamm}\end{align}
with $\sigma_1 = -\frac{1}{8}-r$ and $\sigma_3 = \frac{1}{192}+\frac{1}{6}r^2$ and $\sigma_5=\frac{-1}{640}-\frac{1}{30}r^2-\frac{1}{15}r^3$ and $r=R(1-R)$.
Here, $R=1/6$ yields $\sigma_1 = -19/72$, $\sigma_3 = 131/15552$ and $\sigma_5=-3337/1399680$.

Truncating one of the 36 rational hypergeometric series for $1/\pi$ of the form (\ref{typersz}) results in approximations $\widehat \pi_n$ of $\pi$. Using (\ref{pochhamm}) instead of Lemma \ref{lem2b} gives the following estimate of the truncation error:
\begin{align}
    \left|\widehat\pi_n-\pi\right| \approx \frac{P\sin(\pi R)\sqrt{\pi}}{(1-Z)\sqrt{n}}\cdot |Z|^n.
\end{align}
For asymptotic expansions of these 36 series see Appendix \ref{AppSeries}.

We conclude our paper with another example -- Ramanujan's series (\ref{rama396}) from the introduction:
For this series, we obtain the asymptotic expansion
\begin{align}
    \sum_{k=n}^\infty \tilde s_k=\frac{\tilde s_n}{1-99^{-4}}\cdot\left(1-\frac{1}{\numprint{192119200}\, n}+\frac{\numprint{248104223}}{\numprint{36909787008640000}\, n^2}-\frac{\tilde e_n}{n^3}\right),
\end{align}
with $3.558\cdot 10^{-9}<\tilde e_n<7.463\cdot 10^{-9}$.
Thus Ramanujan's famous series (\ref{rama396}) results in approximations $\tilde{\pi}_n$ of $\pi$ with an asymptotic expansion
\begin{align}
\left|\tilde \pi_n-\pi\right|=99^{-4n}\,\frac{9801\sqrt{\pi}}{1820\sqrt{n}}
\,\exp\mathopen{}\left(-\frac{\numprint{1793359}}{\numprint{6624800}\,n}-\frac{\numprint{15333610991}}{\numprint{17555190016000}\,n^2}+\frac{\tilde\delta_n}{n^3}\right)\mathclose{},    
\end{align}
where $0.0090<\tilde\delta_n<0.0111$.
In \cite[Theorem 3]{mortici}, C.~Mortici has proven similar bounds for another of Ramanujan's series.

Finally, we observe that the exponent $A_1/n+A_2/n^2+\delta_n/n^3$ in Theorem~\ref{theo02} is negative, which implies the weaker estimate for the truncation error in the Chudnovskys' series (\ref{chudseries})
\begin{align}
|\pi_n-\pi|<\frac{\numprint{106720}\sqrt{\numprint{10005}\pi}}{\numprint{1672209}\sqrt{n}}\cdot53360^{-3n}<\frac{11.315}{\sqrt{n}}\cdot53360^{-3n},    
\end{align}
which is valid for all $n\geq 1$. This much weaker estimate is sufficient to determine the number of terms needed to compute a given number of digits of $\pi$ -- even when computing trillions of digits.
But the methods described in this paper demonstrate how to establish precise error bounds for the approximation of $\pi$ through Ramanujan-like series for $1/\pi$
-- and they can be applied to various other generalized hypergeometric series.

  \appendix
  
  \section{Asymptotic expansion of a Pochhammer product}\label{appmorti}
  \renewcommand{\leftmark}{Asymptotic expansion of a Pochhammer product}
  In this appendix, we prove (\ref{pochhamm}) using the method presented by C.~Mortici in \cite{morticiAMM}.
\begin{lemma}\label{lemsti}
For all $R$ with $0<R<1$ it holds 
$$\frac{\left(\frac{1}{2}\right)_n \left(R\right)_n \left(1-R\right)_n}{n!^3}
=\frac{\sin(\pi R)}{(\pi n)^{3/2}}\left(1+\mathcal{O}\left(\frac{1}{n}\right)\right).$$
\end{lemma}
\begin{proof}The definition $(a)_n=\Gamma(n+a)/\Gamma(a)$ yields
\begin{align*}
\frac{\left(\frac{1}{2}\right)_n \left(R\right)_n \left(1-R\right)_n}{n!^3}
  &= \underbrace{\frac{\Gamma(n+\frac{1}{2})\Gamma(n+R)\Gamma(n+1-R)}{\Gamma(n+1)^3}}_{=A}
  \cdot
  \underbrace{\frac{1}{\Gamma(\frac{1}{2})\Gamma(R)\Gamma(1-R)}}_{= B}.
\end{align*}
From Stirling's formula $\Gamma(n+1)=n!\approx\sqrt{2\pi n}\left(\frac{n}{e}\right)^n$ we obtain 
\begin{align*}
\frac{\Gamma(n+1+x)}{n^x\Gamma(n+1)} 
 &\approx \frac{\sqrt{2\pi (n+x)}\left(\frac{n+x}{e}\right)^{n+x}}{n^x\sqrt{2\pi n}\left(\frac{n}{e}\right)^n}
  =       \left(1+\frac{x}{n}\right)^n e^{-x}\left(1+\frac{x}{n}\right)^{x+\frac{1}{2}}
  \approx 1
\end{align*}
which proves $A\approx n^{-3/2}$. Using the reflection formula $\Gamma(R)\Gamma(1-R)=\frac{\pi}{\sin(\pi R)}$ of the $\Gamma$ function yields $\Gamma(\frac{1}{2})=\sqrt{\pi}$ and $B=\frac{\sin(\pi R)}{\pi^{3/2}}$ which proves the Lemma.
\end{proof}

\begin{lemma} For all $n\geq 1$ and all $0<R<1$ it holds
$$\exp\left(\frac{\sigma_1}{n}+\frac{\sigma_3}{n^3}+\frac{\sigma_5}{n^5}\right)
< \frac{\left(\frac{1}{2}\right)_n \left(R\right)_n \left(1-R\right)_n}{n!^3}\cdot\frac{(\pi n)^{3/2}}{\sin(\pi R)}
< \exp\left(\frac{\sigma_1}{n}+\frac{\sigma_3}{n^3}\right)$$
with $\sigma_1 = -\frac{1}{8}-r$ and $\sigma_3 = \frac{1}{192}+\frac{1}{6}r^2$ and $\sigma_5=\frac{-1}{640}-\frac{1}{30}r^2-\frac{1}{15}r^3$ and $r=R(1-R)$.
\end{lemma}
\begin{proof}
First we prove the upper bound and define $r_n$ as follows:
\begin{align*}
    \frac{\left(\frac{1}{2}\right)_n \left(R\right)_n \left(1-R\right)_n}{n!^3} &= 
  \frac{\sin(\pi R)}{(\pi n)^{3/2}}\exp\left(\frac{\sigma_1}{n}+\frac{\sigma_3}{n^3}+r_n\right)\\
\Longrightarrow\quad r_{n} &= \ln\left(\frac{\left(\frac{1}{2}\right)_n \left(R\right)_n \left(1-R\right)_n}{n!^3}\cdot\frac{(\pi n)^{3/2}}{\sin(\pi R)}\right)-\frac{\sigma_1}{n}-\frac{\sigma_3}{n^3}.
\end{align*}
Here, Lemma \ref{lemsti} proves that $r_n$ approaches $0$ for $n\rightarrow\infty$.

Next, we compute $r_n-r_{n-1}$ as a power series $\sum_{k=1}^\infty \varrho_k/n^k$:
\begin{align*}
    r_n-r_{n-1} &= \ln\left(\frac{\left(n-\frac{1}{2}\right) \left(n+R-1\right) \left(n-R\right)}{n^3}\cdot\frac{n^{3/2}}{(n-1)^{3/2}}\right)\\
    &~~~~-\frac{\sigma_1}{n}-\frac{\sigma_3}{n^3}+\frac{\sigma_1}{n-1}+\frac{\sigma_3}{(n-1)^3}\\
    &=\ln\left(\frac{\left(1-\frac{1}{2n}\right) \left(1-\frac{1-R}{n}\right) \left(1-\frac{R}{n}\right)}{\left(1-\frac{1}{n}\right)^{3/2}}\right)
    +\sum_{k=2}^\infty\frac{\sigma_1}{n^k}+\sum_{k=4}^\infty\frac{\binom{k-1}{2}\sigma_3}{n^k}\\
    &=\sum_{k=1}^\infty\frac{\frac{3}{2}-\left(\frac{1}{2}\right)^k-(1-R)^k-R^k}{k n^k}
    +\sum_{k=2}^\infty\frac{\sigma_1}{n^k}+\sum_{k=4}^\infty\frac{\binom{k-1}{2}\sigma_3}{n^k}=:\sum_{k=1}^\infty \frac{\varrho_k}{n^k}
\end{align*}
Here we observe $\varrho_1 = 0$ and $\varrho_2 = \sigma_1+\frac{1}{8}+R(1-R)$. This explains our choice of $\sigma_1$ and proves $\varrho_2=\varrho_3=0$ and $\varrho_4=-\frac{1}{64}-\frac{1}{2}r^2+3\cdot \sigma_3$. By our choice of $\sigma_3$, we have $\varrho_4=\varrho_5=0$. This proves
\begin{align*}
    r_n-r_{n-1} = \sum_{k=6}^\infty \frac{\varrho_k}{n^k}\qquad\text{with}\qquad\varrho_6&=\frac{1}{128}+\frac{1}{6}r^2+\frac{1}{3}r^3.
\end{align*}
In Lemma \ref{lemrho} we prove $\varrho_k>0$ for all $k\geq 6$ and thus $r_n-r_{n-1}>0$.
This proves that $r_n$ is \emph{increasing} to zero, thus $r_n<0$.

For the lower bound, we use $\tilde r_n = r_n-\sigma_5/n^5$:
\begin{align*}
    \tilde r_n-\tilde r_{n-1} &= \frac{\sigma_5}{(n-1)^5}-\frac{\sigma_5}{n^5}+\sum_{k=6}^\infty \frac{\varrho_k}{n^k}
    = \sum_{k=6}^\infty \frac{\varrho_k+\binom{k-1}{4}\sigma_5}{n^k} =: \sum_{k=6}^\infty \frac{\tilde \varrho_k}{n^k}
\end{align*}
Here, our choice of $\sigma_5=-\varrho_6/5$ yields $\tilde \varrho_6=\tilde \varrho_7=0$ and
\begin{align*}
    \tilde r_n-\tilde r_{n-1} = \sum_{k=8}^\infty \frac{\tilde\varrho_k}{n^k}\qquad\text{with}\qquad\tilde\varrho_8 &= - \frac{17}{2048}-\frac{1}{6}r^2-\frac{1}{3}r^3-\frac{1}{4}r^4.
\end{align*}
In Lemma \ref{lemrhotilde} we prove $\tilde\varrho_k<0$ for all $k\geq 8$ and thus $r_n-r_{n-1}<0$.
This proves that $\tilde r_n$ is \emph{decreasing} to zero, which proves $\tilde r_n>0$.
\end{proof}

\begin{lemma}\label{lemrho}
For $k\geq 6$ and for all $0\leq R \leq 1$ it holds $\varrho_k>0$ with
$$\varrho_k = \frac{\frac{1}{2}-\left(\frac{1}{2}\right)^k}{k} + \frac{R-R^k}{k} + \frac{(1-R)-(1-R)^k}{k} +\sigma_1 + \tbinom{k-1}{2}\sigma_3$$
and $\sigma_1 = -\frac{1}{8}-r$ and $\sigma_3 = \frac{1}{192}+\frac{1}{6}r^2$ and $r=R(1-R)$.
\end{lemma}
\begin{proof}
The first three terms are positive. For $k\geq 10$ it holds $\tbinom{k-1}{2}\geq 36$ and
$$\sigma_1 + \tbinom{k-1}{2}\sigma_3\geq -\frac{1}{8}-r + 36(\frac{1}{192}+\frac{1}{6}r^2)=\frac{1}{16}+6r^2-r>0.$$
It remains to check $\varrho_k>0$ for $k<10$: $\varrho_6=\frac{1}{128}+\frac{1}{6}r^2+\frac{1}{3}r^3$, $\varrho_7 =\frac{3}{128}+\frac{1}{2}r^2+r^3$, $\varrho_8 =\frac{95}{2048}+r^2+2r^3-\frac{1}{4}r^4$ and $\varrho_9 =\frac{39}{512}+\frac{5}{3}r^2+\frac{10}{3}r^3-r^4$ are also positive.
\end{proof}


\begin{lemma}\label{lemrhotilde}
For $k\geq 8$ and for all $0\leq R \leq 1$ it holds $\tilde\varrho_k<0$ with
$$\tilde\varrho_k = \frac{\frac{1}{2}-\left(\frac{1}{2}\right)^k}{k} + \frac{R-R^k}{k} + \frac{(1-R)-(1-R)^k}{k} +\sigma_1 + \tbinom{k-1}{2}\sigma_3 + \tbinom{k-1}{4}\sigma_5$$
and $\sigma_1 = -\frac{1}{8}-r$ and $\sigma_3 = \frac{1}{192}+\frac{1}{6}r^2$ and $\sigma_5=\frac{-1}{640}-\frac{1}{30}r^2-\frac{1}{15}r^3$ and $r=R(1-R)$.
\end{lemma}
\begin{proof}
Since $\tbinom{k-1}{4}=\frac{(k-3)(k-4)}{12}\tbinom{k-1}{2}$, for $k\geq 12$ we have $\tbinom{k-1}{4}\geq 6 \tbinom{k-1}{2}$ and
\begin{align*}
    \tbinom{k-1}{2}\sigma_3+\tbinom{k-1}{4}\sigma_5
    &\leq\tbinom{k-1}{2}\left(\sigma_3+6\sigma_5\right) \leq 55\left( -\frac{1}{240}-\frac{1}{30}r^2-\frac{2}{5}r^3\right)\leq-\frac{11}{48}
\end{align*}
The first three terms in $\tilde\varrho_k$ are less than $\frac{3/2}{k}$. This proves $\tilde\varrho_k\leq \frac{3/2}{12}+\sigma_1-\frac{11}{48}<0$ (if $k\geq 12$).

It remains to check $\tilde\varrho_k<0$ for $k<12$:
\begin{align*}
    \tilde\varrho_{8}&=\varrho_8+35\sigma_5
    = -\frac{17}{2048} - \frac{1}{6}r^2 - \frac{1}{3}r^3 - \frac{1}{4}r^4\\
    \tilde\varrho_{9} &=\varrho_9+70\sigma_5
    = -\frac{17}{512} - \frac{2}{3} r^2 - \frac{4}{3} r^3 - r^4\\
    \tilde\varrho_{10} &=\frac{2r^5-25r^4 + 50r^3-35r^2+10r+\frac{511}{1024}}{10}+\sigma_1+36\sigma_3+126\sigma_5\\
    &= 
    \frac{-173}{2048}-\frac{17}{10}r^2-\frac{17}{5}r^3-\frac{5}{2}r^4+\frac{1}{5}r^5\\
    \tilde\varrho_{11} &=\frac{  11r^5-55r^4+77r^3-44r^2+ 11r + \frac{1023}{2048} }{11}+\sigma_1+45\sigma_3+210\sigma_5\\
    &= \frac{-355}{2048}-\frac{7}{2}r^2-7r^3-5r^4+r^5
\end{align*}
Since we have $0\leq R \leq 1$, we have $0\leq r \leq 1$, which proves $\tilde\varrho_k<0$ for all $k$.
\end{proof}

  \section{Asymptotic expansions for further series}
  \renewcommand{\leftmark}{Asymptotic expansions for further series}\label{AppSeries}
  We have used Lemma \ref{lemma-pin} to estimate $\pi_n\approx\pi$ in
$$\pi_n-\pi=\pi~\pi_n~\left(\frac{1}{\pi}-\frac{1}{\pi_n}\right)\approx\pi^2~\left(\frac{1}{\pi}-\frac{1}{\pi_n}\right).$$
This has worked well, because the Chudnovskys' Series \ref{ser7} and Ramanujan's Series \ref{ser23} converge so fast. For most other series, this estimate is valid only for larger $n$. For example in Series \ref{ser1}, we need $n\geq 25$ (see Fig. \ref{fig:Theta1}).

In Series \ref{ser33}, the convergence is too slow and we can not use this estimate at all. This is why we only give an expansion for $\frac{1}{\pi}-\frac{1}{\pi_{33}(n)}$ at Series \ref{ser33}.

\begin{table}[ht]
    \centering
\begin{tabular}{c|c|c|c|c}
\textbf{\#} & \textbf{R} & \textbf{S} & \textbf{Z} & \textbf{P} \\
\hline
\ref{ser1} & $1/6$ & $8/63$ & $-64/125$ & $21\sqrt{15}/25$ \\
\ref{ser2} & $1/6$ & $15/154$ & $-27/512$ & $77\sqrt{2}/32$ \\
\ref{ser3} & $1/6$ & $25/342$ & $-1/512$ & $57\sqrt{6}/32$ \\
\ref{ser4} & $1/6$ & $31/506$ & $-9/64000$ & $759\sqrt{30}/800$ \\
\ref{ser5} & $1/6$ & $263/5418$ & $-1/512000$ & $2709\sqrt{15}/1600$ \\
\ref{ser6} & $1/6$ & $10177/261702$ & $-1/85184000$ & $43617\sqrt{330}/96800$ \\
\ref{ser7} & $1/6$ & $13591409/545140134$ & $-1/53360^3$ & $90856689\sqrt{10005}/711822400$ \\
\ref{ser8} & $1/6$ & $3/28$ & $27/125$ & $28\sqrt{5}/25$ \\
\ref{ser9} & $1/6$ & $1/11$ & $4/125$ & $22\sqrt{15}/25$ \\
\ref{ser10} & $1/6$ & $5/63$ & $8/1331$ & $84\sqrt{33}/121$ \\
\ref{ser11} & $1/6$ & $8/133$ & $64/614125$ & $2394\sqrt{255}/7225$ \\
\hdashline
\ref{ser12} & $1/4$ & $3/20$ & $-1/4$ & $5/2$ \\
\ref{ser13} & $1/4$ & $23/260$ & $-1/324$ & $65/18$ \\
\ref{ser14} & $1/4$ & $1123/21460$ & $-1/777924$ & $5365/882$ \\
\ref{ser15} & $1/4$ & $8/65$ & $-256/3969$ & $65\sqrt{7}/63$ \\
\ref{ser16} & $1/4$ & $3/28$ & $-1/48$ & $7\sqrt{3}/4$ \\
\ref{ser17} & $1/4$ & $41/644$ & $-1/25920$ & $161\sqrt{5}/72$ \\
\ref{ser18} & $1/4$ & $1/7$ & $32/81$ & $14/9$ \\
\ref{ser19} & $1/4$ & $1/8$ & $1/9$ & $4\sqrt{3}/3$ \\
\ref{ser20} & $1/4$ & $1/10$ & $1/81$ & $20\sqrt{2}/9$ \\
\ref{ser21} & $1/4$ & $3/40$ & $1/2401$ & $120\sqrt{3}/49$ \\
\ref{ser22} & $1/4$ & $19/280$ & $1/9801$ & $140\sqrt{11}/99$ \\
\ref{ser23} & $1/4$ & $1103/26390$ & $1/96059601$ & $52780\sqrt{2}/9801$ \\
\hdashline
\ref{ser24} & $1/3$ & $7/51$ & $-1/16$ & $17\sqrt{3}/12$ \\
\ref{ser25} & $1/3$ & $53/615$ & $-1/1024$ & $205\sqrt{3}/96$ \\
\ref{ser26} & $1/3$ & $827/14151$ & $-1/250000$ & $4717\sqrt{3}/1500$ \\
\ref{ser27} & $1/3$ & $1/5$ & $-9/16$ & $5\sqrt{3}/4$ \\
\ref{ser28} & $1/3$ & $1/9$ & $-1/80$ & $3\sqrt{15}/4$ \\
\ref{ser29} & $1/3$ & $13/165$ & $-1/3024$ & $55\sqrt{7}/36$ \\
\ref{ser30} & $1/3$ & $1/6$ & $1/2$ & $2\sqrt{3}/3$ \\
\ref{ser31} & $1/3$ & $2/15$ & $2/27$ & $20/9$ \\
\ref{ser32} & $1/3$ & $4/33$ & $4/125$ & $22\sqrt{3}/15$ \\
\hdashline
\ref{ser33} & $1/2$ & $1/4$ & $-1$ & $2$ \\
\ref{ser34} & $1/2$ & $1/6$ & $-1/8$ & $3\sqrt{2}/2$ \\
\ref{ser35} & $1/2$ & $1/6$ & $1/4$ & $3/2$ \\
\ref{ser36} & $1/2$ & $5/42$ & $1/64$ & $21/8$ \\
\end{tabular}
\vspace{6pt}
    \caption{Coefficients $(R,S,Z,P)$ in formulae of type (\ref{typersz}).}
    \label{tab:rsz}
\end{table}

\begin{series}\label{ser1}
If we denote by $1/\pi_{1}(n)$ the truncated series
\begin{align*}
\frac{1}{\pi_{1}(n)} = \frac{21\sqrt{15}}{25}\sum_{k=0}^{n-1} \frac{\left(\frac{1}{2}\right)_k\left(\frac{1}{6}\right)_k\left(\frac{5}{6}\right)_k}{(k!)^3}\left(k+\frac{8}{63}\right)\left(\frac{-64}{125}\right)^k,
\end{align*}
then we have that $\pi_{1}(n)\rightarrow\pi$ as $n\rightarrow\infty$ and we have the asymptotic expansion
\begin{dmath*}
\pi_{1}(n)-\pi=\left(\frac{-64}{125}\right)^n\frac{5\sqrt{15\pi}}{18\sqrt{n}}\cdot\exp\left(
\frac{7}{216n}
+\frac{-80}{729n^2}
+\frac{23129}{3779136n^3}
+\frac{47168}{531441n^4}
\Theta_{1}(n)\right).
\end{dmath*}
For all $n\geq 25$ we have that $0<\Theta_{1}(n)<1$
and for $n\rightarrow\infty$ we have that $\Theta_{1}(n) \rightarrow 1$.
\end{series}

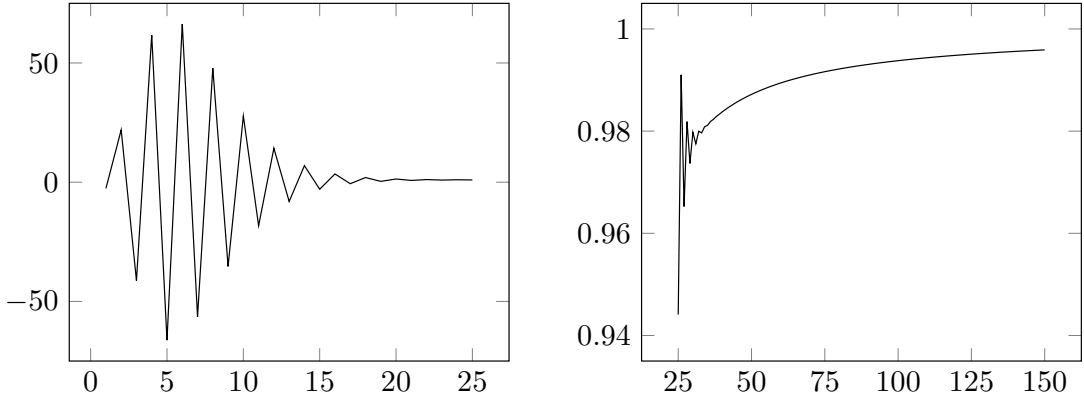
\begin{figure}
\begin{tikzpicture}\begin{axis}[width=6.7cm,ymin=-75,ymax=75]
\addplot[solid]  coordinates { 
(1, -2.6134928679578042)
(2, 21.925520592315696)
(3, -41.22341863639891)
(4, 61.671916921210695)
(5, -66.25578047234418)
(6, 66.34938594349074)
(7, -56.435900677385106)
(8, 47.803842435089884)
(9, -35.321694043919585)
(10, 27.75626181539341)
(11, -18.241199699964064)
(12, 14.248989442021566)
(13, -8.073919056951986)
(14, 6.931839799108738)
(15, -2.948001957145045)
(16, 3.458156092738221)
(17, -0.6257627041388949)
(18, 1.951657280606851)
(19, 0.35044918121494806)
(20, 1.3408060041599688)
(21, 0.7391601837657914)
(22, 1.106220624896502)
(23, 0.8881581861744671)
(24, 1.0205313176791484)
(25, 0.9441056553600823)
};
\end{axis}\end{tikzpicture}\hfill\begin{tikzpicture}\begin{axis}[width=6.7cm,ymin=0.935,ymax=1.005,xtick={25,50,75,100,125,150}]
\addplot[solid]  coordinates { 
(25, 0.9441056553600823)
(26, 0.9910270711997623)
(27, 0.9652431987203423)
(28, 0.9818652296682363)
(29, 0.9736775169238677)
(30, 0.9797635667933279)
(31, 0.9775166372803474)
(32, 0.9799749222917677)
(33, 0.9796615661006075)
(34, 0.9808607365712706)
(35, 0.9811364797498999)
(36, 0.9818786600641717)
(37, 0.9823048055600585)
(38, 0.9828613693817471)
(39, 0.9833007262580754)
(40, 0.9837655409607662)
(41, 0.9841789663931623)
(42, 0.9845868060156849)
(43, 0.9849660346305765)
(44, 0.9853317710349845)
(45, 0.9856776885066614)
(46, 0.9860091832006208)
(47, 0.9863249822123351)
(48, 0.9866273535240605)
(49, 0.9869164743077455)
(50, 0.9871935550270515)
(51, 0.9874591292575089)
(52, 0.9877140059959681)
(53, 0.9879587557208384)
(54, 0.9881940009481969)
(55, 0.9884202644310125)
(56, 0.9886380589058987)
(57, 0.9888478439318534)
(58, 0.9890500547338728)
(59, 0.9892450911821975)
(60, 0.9894333280186183)
(61, 0.9896151130790721)
(62, 0.9897907716665173)
(63, 0.9899606072443107)
(64, 0.9901249038289484)
(65, 0.9902839271958699)
(66, 0.9904379264925594)
(67, 0.9905871354207431)
(68, 0.9907317734637759)
(69, 0.9908720469214508)
(70, 0.9910081498980864)
(71, 0.9911402651805002)
(72, 0.9912685650536565)
(73, 0.9913932120402794)
(74, 0.9915143595825237)
(75, 0.9916321526654883)
(76, 0.991746728391183)
(77, 0.9918582165058972)
(78, 0.9919667398861662)
(79, 0.9920724149866165)
(80, 0.9921753522533714)
(81, 0.9922756565059336)
(82, 0.9923734272903795)
(83, 0.9924687592063053)
(84, 0.9925617422097935)
(85, 0.9926524618944098)
(86, 0.9927409997520708)
(87, 0.992827433415441)
(88, 0.9929118368833659)
(89, 0.9929942807307063)
(90, 0.9930748323038194)
(91, 0.9931535559028168)
(92, 0.99323051295163)
(93, 0.9933057621568228)
(94, 0.9933793596560102)
(95, 0.9934513591566667)
(96, 0.9935218120660416)
(97, 0.9935907676128382)
(98, 0.9936582729612595)
(99, 0.9937243733179727)
(100, 0.9937891120324981)
(101, 0.993852530691491)
(102, 0.9939146692073441)
(103, 0.993975565901504)
(104, 0.9940352575828686)
(105, 0.9940937796215972)
(106, 0.9941511660186464)
(107, 0.9942074494713143)
(108, 0.9942626614350608)
(109, 0.994316832181845)
(110, 0.9943699908552089)
(111, 0.9944221655223162)
(112, 0.9944733832231396)
(113, 0.9945236700169783)
(114, 0.9945730510264722)
(115, 0.99462155047927)
(116, 0.9946691917474934)
(117, 0.9947159973851334)
(118, 0.9947619891635043)
(119, 0.9948071881048718)
(120, 0.9948516145143618)
(121, 0.9948952880102551)
(122, 0.99493822755276)
(123, 0.994980451471351)
(124, 0.9950219774907579)
(125, 0.9950628227556811)
(126, 0.9951030038543052)
(127, 0.9951425368406792)
(128, 0.9951814372560263)
(129, 0.995219720149042)
(130, 0.9952574000952367)
(131, 0.9952944912153739)
(132, 0.9953310071930548)
(133, 0.9953669612914922)
(134, 0.9954023663695195)
(135, 0.9954372348968747)
(136, 0.995471578968795)
(137, 0.9955054103199619)
(138, 0.9955387403378265)
(139, 0.9955715800753503)
(140, 0.9956039402631874)
(141, 0.9956358313213408)
(142, 0.9956672633703145)
(143, 0.9956982462417907)
(144, 0.9957287894888529)
(145, 0.9957589023957782)
(146, 0.9957885939874199)
(147, 0.9958178730381999)
(148, 0.9958467480807298)
(149, 0.9958752274140787)
(150, 0.9959033191117044)
};
\end{axis}\end{tikzpicture}
    \caption{Values of $\Theta_1(n)$ from Series \ref{ser1} showing strong oscillation for $n<25$ (left) and damped oscillation for $n\geq 25$ (right).}
    \label{fig:Theta1}
\end{figure}

\hrule

\begin{series}\label{ser2}
If we denote by $1/\pi_{2}(n)$ the truncated series
\begin{align*}
\frac{1}{\pi_{2}(n)} = \frac{77\sqrt{2}}{32}\sum_{k=0}^{n-1} \frac{\left(\frac{1}{2}\right)_k\left(\frac{1}{6}\right)_k\left(\frac{5}{6}\right)_k}{(k!)^3}\left(k+\frac{15}{154}\right)\left(\frac{-27}{512}\right)^k,
\end{align*}
then we have that $\pi_{2}(n)\rightarrow\pi$ as $n\rightarrow\infty$ and we have the asymptotic expansion
\begin{dmath*}
\pi_{2}(n)-\pi=\left(\frac{-27}{512}\right)^n\frac{8\sqrt{2\pi}}{7\sqrt{n}}\cdot\exp\left(
\frac{-499}{3528n}
+\frac{-291}{9604n^2}
+\frac{57332435}{1829677248n^3}
+\frac{-2919207}{184473632n^4}
\Theta_{2}(n)\right).
\end{dmath*}
For all $n\geq 4$ we have that $0<\Theta_{2}(n)<1$
and for $n\rightarrow\infty$ we have that $\Theta_{2}(n) \rightarrow 1$.
\end{series}
\hrule

\begin{series}\label{ser3}
If we denote by $1/\pi_{3}(n)$ the truncated series
\begin{align*}
\frac{1}{\pi_{3}(n)} = \frac{57\sqrt{6}}{32}\sum_{k=0}^{n-1} \frac{\left(\frac{1}{2}\right)_k\left(\frac{1}{6}\right)_k\left(\frac{5}{6}\right)_k}{(k!)^3}\left(k+\frac{25}{342}\right)\left(\frac{-1}{512}\right)^k,
\end{align*}
then we have that $\pi_{3}(n)\rightarrow\pi$ as $n\rightarrow\infty$ and we have the asymptotic expansion
\begin{dmath*}
\pi_{3}(n)-\pi=\left(\frac{-1}{512}\right)^n\frac{8\sqrt{6\pi}}{9\sqrt{n}}\cdot\exp\left(
\frac{-41}{216n}
+\frac{-11}{2916n^2}
+\frac{36665}{3779136n^3}
\Theta_{3}(n)\right).
\end{dmath*}
For all $n\geq 1$ we have that $0<\Theta_{3}(n)<1$
and for $n\rightarrow\infty$ we have that $\Theta_{3}(n) \rightarrow 1$.
\end{series}
\hrule

\begin{series}\label{ser4}
If we denote by $1/\pi_{4}(n)$ the truncated series
\begin{align*}
\frac{1}{\pi_{4}(n)} = \frac{759\sqrt{30}}{800}\sum_{k=0}^{n-1} \frac{\left(\frac{1}{2}\right)_k\left(\frac{1}{6}\right)_k\left(\frac{5}{6}\right)_k}{(k!)^3}\left(k+\frac{31}{506}\right)\left(\frac{-9}{64000}\right)^k,
\end{align*}
then we have that $\pi_{4}(n)\rightarrow\pi$ as $n\rightarrow\infty$ and we have the asymptotic expansion
\begin{dmath*}
\pi_{4}(n)-\pi=\left(\frac{-9}{64000}\right)^n\frac{120\sqrt{30\pi}}{253\sqrt{n}}\cdot\exp\left(
\frac{-933499}{4608648n}
+\frac{-32087155}{16388608324n^2}
+\frac{35019155487209387}{4078583656660041408n^3}
\Theta_{4}(n)\right).
\end{dmath*}
For all $n\geq 1$ we have that $0<\Theta_{4}(n)<1$
and for $n\rightarrow\infty$ we have that $\Theta_{4}(n) \rightarrow 1$.
\end{series}
\hrule

\begin{series}\label{ser5}
If we denote by $1/\pi_{5}(n)$ the truncated series
\begin{align*}
\frac{1}{\pi_{5}(n)} = \frac{2709\sqrt{15}}{1600}\sum_{k=0}^{n-1} \frac{\left(\frac{1}{2}\right)_k\left(\frac{1}{6}\right)_k\left(\frac{5}{6}\right)_k}{(k!)^3}\left(k+\frac{263}{5418}\right)\left(\frac{-1}{512000}\right)^k,
\end{align*}
then we have that $\pi_{5}(n)\rightarrow\pi$ as $n\rightarrow\infty$ and we have the asymptotic expansion
\begin{dmath*}
\pi_{5}(n)-\pi=\left(\frac{-1}{512000}\right)^n\frac{160\sqrt{15\pi}}{189\sqrt{n}}\cdot\exp\left(
\frac{-20513}{95256n}
+\frac{-668795}{567106596n^2}
+\frac{2742951440609}{324121835451456n^3}
\Theta_{5}(n)\right).
\end{dmath*}
For all $n\geq 1$ we have that $0<\Theta_{5}(n)<1$
and for $n\rightarrow\infty$ we have that $\Theta_{5}(n) \rightarrow 1$.
\end{series}
\hrule

\begin{series}\label{ser6}
If we denote by $1/\pi_{6}(n)$ the truncated series
\begin{align*}
\frac{1}{\pi_{6}(n)} = \frac{43617\sqrt{330}}{96800}\sum_{k=0}^{n-1} \frac{\left(\frac{1}{2}\right)_k\left(\frac{1}{6}\right)_k\left(\frac{5}{6}\right)_k}{(k!)^3}\left(k+\frac{10177}{261702}\right)\left(\frac{-1}{440^3}\right)^k,
\end{align*}
then we have that $\pi_{6}(n)\rightarrow\pi$ as $n\rightarrow\infty$ and we have the asymptotic expansion
\begin{dmath*}
\pi_{6}(n)-\pi=\left(\frac{-1}{440^3}\right)^n\frac{440\sqrt{330\pi}}{1953\sqrt{n}}\cdot\exp\left(
\frac{-2288537}{10171224n}
+\frac{-4889066795}{6465862353636n^2}
+\frac{3331546471820293001}{394594406111993822784n^3}
\Theta_{6}(n)\right).
\end{dmath*}
For all $n\geq 1$ we have that $0<\Theta_{6}(n)<1$
and for $n\rightarrow\infty$ we have that $\Theta_{6}(n) \rightarrow 1$.
\end{series}
\hrule

\begin{series}\label{ser7}
If we denote by $1/\pi_{7}(n)$ the truncated series
\begin{align*}
\frac{1}{\pi_{7}(n)} = \frac{90856689\sqrt{10005}}{711822400}\sum_{k=0}^{n-1} \frac{\left(\frac{1}{2}\right)_k\left(\frac{1}{6}\right)_k\left(\frac{5}{6}\right)_k}{(k!)^3}\left(k+\frac{13591409}{545140134}\right)\left(\frac{-1}{53360^3}\right)^k,
\end{align*}
then we have that $\pi_{7}(n)\rightarrow\pi$ as $n\rightarrow\infty$ and we have the asymptotic expansion
\begin{dmath*}
\pi_{7}(n)-\pi=\left(\frac{-1}{53360^3}\right)^n\frac{106720\sqrt{10005\pi}}{1672209\sqrt{n}}\cdot\exp\left(
\frac{-1781843197433}{7456754505816n}
+\frac{-1080096011925710088395}{3475199235000451148614116n^2}
+\frac{1310485187935583963485460802564780329}{155482245325187582131326612761170191936n^3}
\Theta_{7}(n)\right).
\end{dmath*}
For all $n\geq 1$ we have that $0<\Theta_{7}(n)<1$
and for $n\rightarrow\infty$ we have that $\Theta_{7}(n) \rightarrow 1$.
\end{series}
\hrule

\begin{series}\label{ser8}
If we denote by $1/\pi_{8}(n)$ the truncated series
\begin{align*}
\frac{1}{\pi_{8}(n)} = \frac{28\sqrt{5}}{25}\sum_{k=0}^{n-1} \frac{\left(\frac{1}{2}\right)_k\left(\frac{1}{6}\right)_k\left(\frac{5}{6}\right)_k}{(k!)^3}\left(k+\frac{3}{28}\right)\left(\frac{27}{125}\right)^k,
\end{align*}
then we have that $\pi_{8}(n)\rightarrow\pi$ as $n\rightarrow\infty$ and we have the asymptotic expansion
\begin{dmath*}
\pi_{8}(n)-\pi=\left(\frac{27}{125}\right)^n\frac{5\sqrt{5\pi}}{7\sqrt{n}}\cdot\exp\left(
\frac{-1039}{3528n}
+\frac{3615}{19208n^2}
+\frac{-601881853}{1829677248n^3}
+\frac{288589047}{368947264n^4}
+\frac{-132250237090171}{56482136645760n^5}
\Theta_{8}(n)\right).
\end{dmath*}
For all $n\geq 4$ we have that $0<\Theta_{8}(n)<1$
and for $n\rightarrow\infty$ we have that $\Theta_{8}(n) \rightarrow 1$.
\end{series}
\hrule

\begin{series}\label{ser9}
If we denote by $1/\pi_{9}(n)$ the truncated series
\begin{align*}
\frac{1}{\pi_{9}(n)} = \frac{22\sqrt{15}}{25}\sum_{k=0}^{n-1} \frac{\left(\frac{1}{2}\right)_k\left(\frac{1}{6}\right)_k\left(\frac{5}{6}\right)_k}{(k!)^3}\left(k+\frac{1}{11}\right)\left(\frac{4}{125}\right)^k,
\end{align*}
then we have that $\pi_{9}(n)\rightarrow\pi$ as $n\rightarrow\infty$ and we have the asymptotic expansion
\begin{dmath*}
\pi_{9}(n)-\pi=\left(\frac{4}{125}\right)^n\frac{5\sqrt{15\pi}}{11\sqrt{n}}\cdot\exp\left(
\frac{-1651}{8712n}
+\frac{3865}{263538n^2}
+\frac{-339804517}{27551316672n^3}
+\frac{1739500937}{69452277444n^4}
+\frac{-119121688877123}{3300372224138880n^5}
\Theta_{9}(n)\right).
\end{dmath*}
For all $n\geq 2$ we have that $0<\Theta_{9}(n)<1$
and for $n\rightarrow\infty$ we have that $\Theta_{9}(n) \rightarrow 1$.
\end{series}
\hrule

\begin{series}\label{ser10}
If we denote by $1/\pi_{10}(n)$ the truncated series
\begin{align*}
\frac{1}{\pi_{10}(n)} = \frac{84\sqrt{33}}{121}\sum_{k=0}^{n-1} \frac{\left(\frac{1}{2}\right)_k\left(\frac{1}{6}\right)_k\left(\frac{5}{6}\right)_k}{(k!)^3}\left(k+\frac{5}{63}\right)\left(\frac{8}{1331}\right)^k,
\end{align*}
then we have that $\pi_{10}(n)\rightarrow\pi$ as $n\rightarrow\infty$ and we have the asymptotic expansion
\begin{dmath*}
\pi_{10}(n)-\pi=\left(\frac{8}{1331}\right)^n\frac{22\sqrt{33\pi}}{63\sqrt{n}}\cdot\exp\left(
\frac{-1985}{10584n}
+\frac{899}{3500658n^2}
+\frac{2216760425}{444611571264n^3}
+\frac{46674125279}{12254606432964n^4}
+\frac{-483110468975635}{74115859706566272n^5}
\Theta_{10}(n)\right).
\end{dmath*}
For all $n\geq 2$ we have that $0<\Theta_{10}(n)<1$
and for $n\rightarrow\infty$ we have that $\Theta_{10}(n) \rightarrow 1$.
\end{series}
\hrule

\begin{series}\label{ser11}
If we denote by $1/\pi_{11}(n)$ the truncated series
\begin{align*}
\frac{1}{\pi_{11}(n)} = \frac{2394\sqrt{255}}{7225}\sum_{k=0}^{n-1} \frac{\left(\frac{1}{2}\right)_k\left(\frac{1}{6}\right)_k\left(\frac{5}{6}\right)_k}{(k!)^3}\left(k+\frac{8}{133}\right)\left(\frac{64}{85^3}\right)^k,
\end{align*}
then we have that $\pi_{11}(n)\rightarrow\pi$ as $n\rightarrow\infty$ and we have the asymptotic expansion
\begin{dmath*}
\pi_{11}(n)-\pi=\left(\frac{64}{85^3}\right)^n\frac{85\sqrt{255\pi}}{513\sqrt{n}}\cdot\exp\left(
\frac{-143017}{701784n}
+\frac{-13456880}{7695324729n^2}
+\frac{1092857149734121}{129610938470796864n^3}
+\frac{3781366699871168}{59218022684758923441n^4}
+\frac{-8157754659852064155749017}{3324660866544068506249499520n^5}
\Theta_{11}(n)\right).
\end{dmath*}
For all $n\geq 1$ we have that $0<\Theta_{11}(n)<1$
and for $n\rightarrow\infty$ we have that $\Theta_{11}(n) \rightarrow 1$.
\end{series}
\hrule

\begin{series}\label{ser12}
If we denote by $1/\pi_{12}(n)$ the truncated series
\begin{align*}
\frac{1}{\pi_{12}(n)} = \frac{5}{2}\sum_{k=0}^{n-1} \frac{\left(\frac{1}{2}\right)_k\left(\frac{1}{4}\right)_k\left(\frac{3}{4}\right)_k}{(k!)^3}\left(k+\frac{3}{20}\right)\left(\frac{-1}{4}\right)^k,
\end{align*}
then we have that $\pi_{12}(n)\rightarrow\pi$ as $n\rightarrow\infty$ and we have the asymptotic expansion
\begin{dmath*}
\pi_{12}(n)-\pi=\left(\frac{-1}{4}\right)^n\frac{\sqrt{2\pi}}{\sqrt{n}}\cdot\exp\left(
\frac{-1}{16n}
+\frac{-3}{32n^2}
+\frac{73}{1536n^3}
+\frac{33}{1024n^4}
+\frac{-1861}{20480n^5}
\Theta_{12}(n)\right).
\end{dmath*}
For all $n\geq 12$ we have that $0<\Theta_{12}(n)<1$
and for $n\rightarrow\infty$ we have that $\Theta_{12}(n) \rightarrow 1$.
\end{series}
\hrule

\begin{series}\label{ser13}
If we denote by $1/\pi_{13}(n)$ the truncated series
\begin{align*}
\frac{1}{\pi_{13}(n)} = \frac{65}{18}\sum_{k=0}^{n-1} \frac{\left(\frac{1}{2}\right)_k\left(\frac{1}{4}\right)_k\left(\frac{3}{4}\right)_k}{(k!)^3}\left(k+\frac{23}{260}\right)\left(\frac{-1}{324}\right)^k,
\end{align*}
then we have that $\pi_{13}(n)\rightarrow\pi$ as $n\rightarrow\infty$ and we have the asymptotic expansion
\begin{dmath*}
\pi_{13}(n)-\pi=\left(\frac{-1}{324}\right)^n\frac{9\sqrt{2\pi}}{5\sqrt{n}}\cdot\exp\left(
\frac{-89}{400n}
+\frac{-23}{4000n^2}
+\frac{317969}{24000000n^3}
\Theta_{13}(n)\right).
\end{dmath*}
For all $n\geq 1$ we have that $0<\Theta_{13}(n)<1$
and for $n\rightarrow\infty$ we have that $\Theta_{13}(n) \rightarrow 1$.
\end{series}
\hrule

\begin{series}\label{ser14}
If we denote by $1/\pi_{14}(n)$ the truncated series
\begin{align*}
\frac{1}{\pi_{14}(n)} = \frac{5365}{882}\sum_{k=0}^{n-1} \frac{\left(\frac{1}{2}\right)_k\left(\frac{1}{4}\right)_k\left(\frac{3}{4}\right)_k}{(k!)^3}\left(k+\frac{1123}{21460}\right)\left(\frac{-1}{777924}\right)^k,
\end{align*}
then we have that $\pi_{14}(n)\rightarrow\pi$ as $n\rightarrow\infty$ and we have the asymptotic expansion
\begin{dmath*}
\pi_{14}(n)-\pi=\left(\frac{-1}{777924}\right)^n\frac{441\sqrt{2\pi}}{145\sqrt{n}}\cdot\exp\left(
\frac{-87521}{336400n}
+\frac{-155039}{113164960n^2}
+\frac{158694702263081}{14275759704000000n^3}
\Theta_{14}(n)\right).
\end{dmath*}
For all $n\geq 1$ we have that $0<\Theta_{14}(n)<1$
and for $n\rightarrow\infty$ we have that $\Theta_{14}(n) \rightarrow 1$.
\end{series}
\hrule

\begin{series}\label{ser15}
If we denote by $1/\pi_{15}(n)$ the truncated series
\begin{align*}
\frac{1}{\pi_{15}(n)} = \frac{65\sqrt{7}}{63}\sum_{k=0}^{n-1} \frac{\left(\frac{1}{2}\right)_k\left(\frac{1}{4}\right)_k\left(\frac{3}{4}\right)_k}{(k!)^3}\left(k+\frac{8}{65}\right)\left(\frac{-256}{3969}\right)^k,
\end{align*}
then we have that $\pi_{15}(n)\rightarrow\pi$ as $n\rightarrow\infty$ and we have the asymptotic expansion
\begin{dmath*}
\pi_{15}(n)-\pi=\left(\frac{-256}{3969}\right)^n\frac{63\sqrt{14\pi}}{130\sqrt{n}}\cdot\exp\left(
\frac{-10757}{67600n}
+\frac{-140944}{3570125n^2}
+\frac{4556280871313}{115843416000000n^3}
+\frac{-5541878772448}{318644812890625n^4}
+\frac{-123624675437230027457}{27571698369800000000000n^5}
\Theta_{15}(n)\right).
\end{dmath*}
For all $n\geq 5$ we have that $0<\Theta_{15}(n)<1$
and for $n\rightarrow\infty$ we have that $\Theta_{15}(n) \rightarrow 1$.
\end{series}
\hrule

\begin{series}\label{ser16}
If we denote by $1/\pi_{16}(n)$ the truncated series
\begin{align*}
\frac{1}{\pi_{16}(n)} = \frac{7\sqrt{3}}{4}\sum_{k=0}^{n-1} \frac{\left(\frac{1}{2}\right)_k\left(\frac{1}{4}\right)_k\left(\frac{3}{4}\right)_k}{(k!)^3}\left(k+\frac{3}{28}\right)\left(\frac{-1}{48}\right)^k,
\end{align*}
then we have that $\pi_{16}(n)\rightarrow\pi$ as $n\rightarrow\infty$ and we have the asymptotic expansion
\begin{dmath*}
\pi_{16}(n)-\pi=\left(\frac{-1}{48}\right)^n\frac{6\sqrt{6\pi}}{7\sqrt{n}}\cdot\exp\left(
\frac{-153}{784n}
+\frac{-1331}{76832n^2}
+\frac{1385667}{60236288n^3}
+\frac{-61756703}{5903156224n^4}
+\frac{3777837381}{826441871360n^5}
\Theta_{16}(n)\right).
\end{dmath*}
For all $n\geq 3$ we have that $0<\Theta_{16}(n)<1$
and for $n\rightarrow\infty$ we have that $\Theta_{16}(n) \rightarrow 1$.
\end{series}
\hrule

\begin{series}\label{ser17}
If we denote by $1/\pi_{17}(n)$ the truncated series
\begin{align*}
\frac{1}{\pi_{17}(n)} = \frac{161\sqrt{5}}{72}\sum_{k=0}^{n-1} \frac{\left(\frac{1}{2}\right)_k\left(\frac{1}{4}\right)_k\left(\frac{3}{4}\right)_k}{(k!)^3}\left(k+\frac{41}{644}\right)\left(\frac{-1}{25920}\right)^k,
\end{align*}
then we have that $\pi_{17}(n)\rightarrow\pi$ as $n\rightarrow\infty$ and we have the asymptotic expansion
\begin{dmath*}
\pi_{17}(n)-\pi=\left(\frac{-1}{25920}\right)^n\frac{180\sqrt{10\pi}}{161\sqrt{n}}\cdot\exp\left(
\frac{-103193}{414736n}
+\frac{-44090635}{21500743712n^2}
+\frac{299081123487449}{26751397332420096n^3}
\Theta_{17}(n)\right).
\end{dmath*}
For all $n\geq 1$ we have that $0<\Theta_{17}(n)<1$
and for $n\rightarrow\infty$ we have that $\Theta_{17}(n) \rightarrow 1$.
\end{series}
\hrule

\begin{series}\label{ser18}
If we denote by $1/\pi_{18}(n)$ the truncated series
\begin{align*}
\frac{1}{\pi_{18}(n)} = \frac{14}{9}\sum_{k=0}^{n-1} \frac{\left(\frac{1}{2}\right)_k\left(\frac{1}{4}\right)_k\left(\frac{3}{4}\right)_k}{(k!)^3}\left(k+\frac{1}{7}\right)\left(\frac{32}{81}\right)^k,
\end{align*}
then we have that $\pi_{18}(n)\rightarrow\pi$ as $n\rightarrow\infty$ and we have the asymptotic expansion
\begin{dmath*}
\pi_{18}(n)-\pi=\left(\frac{32}{81}\right)^n\frac{9\sqrt{2\pi}}{7\sqrt{n}}\cdot\exp\left(
\frac{-389}{784n}
+\frac{2939}{4802n^2}
+\frac{-302662495}{180708864n^3}
+\frac{148922927}{23059204n^4}
+\frac{-26488619802887}{826441871360n^5}
\Theta_{18}(n)\right).
\end{dmath*}
For all $n\geq 1$ we have that $0<\Theta_{18}(n)<1$
and for $n\rightarrow\infty$ we have that $\Theta_{18}(n) \rightarrow 1$.
\end{series}
\hrule

\begin{series}\label{ser19}
If we denote by $1/\pi_{19}(n)$ the truncated series
\begin{align*}
\frac{1}{\pi_{19}(n)} = \frac{4\sqrt{3}}{3}\sum_{k=0}^{n-1} \frac{\left(\frac{1}{2}\right)_k\left(\frac{1}{4}\right)_k\left(\frac{3}{4}\right)_k}{(k!)^3}\left(k+\frac{1}{8}\right)\left(\frac{1}{9}\right)^k,
\end{align*}
then we have that $\pi_{19}(n)\rightarrow\pi$ as $n\rightarrow\infty$ and we have the asymptotic expansion
\begin{dmath*}
\pi_{19}(n)-\pi=\left(\frac{1}{9}\right)^n\frac{3\sqrt{6\pi}}{4\sqrt{n}}\cdot\exp\left(
\frac{-1}{4n}
+\frac{37}{512n^2}
+\frac{-1195}{12288n^3}
+\frac{46945}{262144n^4}
+\frac{-1956919}{5242880n^5}
\Theta_{19}(n)\right).
\end{dmath*}
For all $n\geq 3$ we have that $0<\Theta_{19}(n)<1$
and for $n\rightarrow\infty$ we have that $\Theta_{19}(n) \rightarrow 1$.
\end{series}
\hrule

\begin{series}\label{ser20}
If we denote by $1/\pi_{20}(n)$ the truncated series
\begin{align*}
\frac{1}{\pi_{20}(n)} = \frac{20\sqrt{2}}{9}\sum_{k=0}^{n-1} \frac{\left(\frac{1}{2}\right)_k\left(\frac{1}{4}\right)_k\left(\frac{3}{4}\right)_k}{(k!)^3}\left(k+\frac{1}{10}\right)\left(\frac{1}{81}\right)^k,
\end{align*}
then we have that $\pi_{20}(n)\rightarrow\pi$ as $n\rightarrow\infty$ and we have the asymptotic expansion
\begin{dmath*}
\pi_{20}(n)-\pi=\left(\frac{1}{81}\right)^n\frac{9\sqrt{\pi}}{4\sqrt{n}}\cdot\exp\left(
\frac{-7}{32n}
+\frac{5}{2048n^2}
+\frac{323}{98304n^3}
+\frac{37313}{4194304n^4}
+\frac{-2253817}{167772160n^5}
\Theta_{20}(n)\right).
\end{dmath*}
For all $n\geq 2$ we have that $0<\Theta_{20}(n)<1$
and for $n\rightarrow\infty$ we have that $\Theta_{20}(n) \rightarrow 1$.
\end{series}
\hrule

\begin{series}\label{ser21}
If we denote by $1/\pi_{21}(n)$ the truncated series
\begin{align*}
\frac{1}{\pi_{21}(n)} = \frac{120\sqrt{3}}{49}\sum_{k=0}^{n-1} \frac{\left(\frac{1}{2}\right)_k\left(\frac{1}{4}\right)_k\left(\frac{3}{4}\right)_k}{(k!)^3}\left(k+\frac{3}{40}\right)\left(\frac{1}{2401}\right)^k,
\end{align*}
then we have that $\pi_{21}(n)\rightarrow\pi$ as $n\rightarrow\infty$ and we have the asymptotic expansion
\begin{dmath*}
\pi_{21}(n)-\pi=\left(\frac{1}{2401}\right)^n\frac{49\sqrt{6\pi}}{40\sqrt{n}}\cdot\exp\left(
\frac{-1141}{4800n}
+\frac{-23567}{9216000n^2}
+\frac{3626854163}{331776000000n^3}
+\frac{606576592193}{2123366400000000n^4}
+\frac{-44351997882926281}{12740198400000000000n^5}
\Theta_{21}(n)\right).
\end{dmath*}
For all $n\geq 1$ we have that $0<\Theta_{21}(n)<1$
and for $n\rightarrow\infty$ we have that $\Theta_{21}(n) \rightarrow 1$.
\end{series}
\hrule

\begin{series}\label{ser22}
If we denote by $1/\pi_{22}(n)$ the truncated series
\begin{align*}
\frac{1}{\pi_{22}(n)} = \frac{140\sqrt{11}}{99}\sum_{k=0}^{n-1} \frac{\left(\frac{1}{2}\right)_k\left(\frac{1}{4}\right)_k\left(\frac{3}{4}\right)_k}{(k!)^3}\left(k+\frac{19}{280}\right)\left(\frac{1}{9801}\right)^k,
\end{align*}
then we have that $\pi_{22}(n)\rightarrow\pi$ as $n\rightarrow\infty$ and we have the asymptotic expansion
\begin{dmath*}
\pi_{22}(n)-\pi=\left(\frac{1}{9801}\right)^n\frac{99\sqrt{22\pi}}{140\sqrt{n}}\cdot\exp\left(
\frac{-1199}{4900n}
+\frac{-344063}{153664000n^2}
+\frac{250800696613}{22588608000000n^3}
+\frac{40239008483873}{590315622400000000n^4}
+\frac{-959530669459838519}{295157811200000000000n^5}
\Theta_{22}(n)\right).
\end{dmath*}
For all $n\geq 1$ we have that $0<\Theta_{22}(n)<1$
and for $n\rightarrow\infty$ we have that $\Theta_{22}(n) \rightarrow 1$.
\end{series}
\hrule

\begin{series}\label{ser23}
If we denote by $1/\pi_{23}(n)$ the truncated series
\begin{align*}
\frac{1}{\pi_{23}(n)} = \frac{52780\sqrt{2}}{9801}\sum_{k=0}^{n-1} \frac{\left(\frac{1}{2}\right)_k\left(\frac{1}{4}\right)_k\left(\frac{3}{4}\right)_k}{(k!)^3}\left(k+\frac{1103}{26390}\right)\left(\frac{1}{99^4}\right)^k,
\end{align*}
then we have that $\pi_{23}(n)\rightarrow\pi$ as $n\rightarrow\infty$ and we have the asymptotic expansion
\begin{dmath*}
\pi_{23}(n)-\pi=\left(\frac{1}{99^4}\right)^n\frac{9801\sqrt{\pi}}{1820\sqrt{n}}\cdot\exp\left(
\frac{-1793359}{6624800n}
+\frac{-15333610991}{17555190016000n^2}
+\frac{9674999636298154667}{872247171134976000000n^3}
+\frac{-5815449554312741373727}{7704617412446652006400000000n^4}
+\frac{-28927901748784728363322170328567}{9114562398924389323571200000000000n^5}
\Theta_{23}(n)\right).
\end{dmath*}
For all $n\geq 1$ we have that $0<\Theta_{23}(n)<1$
and for $n\rightarrow\infty$ we have that $\Theta_{23}(n) \rightarrow 1$.
\end{series}
\hrule

\begin{series}\label{ser24}
If we denote by $1/\pi_{24}(n)$ the truncated series
\begin{align*}
\frac{1}{\pi_{24}(n)} = \frac{17\sqrt{3}}{12}\sum_{k=0}^{n-1} \frac{\left(\frac{1}{2}\right)_k\left(\frac{1}{3}\right)_k\left(\frac{2}{3}\right)_k}{(k!)^3}\left(k+\frac{7}{51}\right)\left(\frac{-1}{16}\right)^k,
\end{align*}
then we have that $\pi_{24}(n)\rightarrow\pi$ as $n\rightarrow\infty$ and we have the asymptotic expansion
\begin{dmath*}
\pi_{24}(n)-\pi=\left(\frac{-1}{16}\right)^n\frac{2\sqrt{\pi}}{\sqrt{n}}\cdot\exp\left(
\frac{-13}{72n}
+\frac{-1}{24n^2}
+\frac{665}{15552n^3}
+\frac{-95}{5184n^4}
+\frac{-6559}{1399680n^5}
\Theta_{24}(n)\right).
\end{dmath*}
For all $n\geq 5$ we have that $0<\Theta_{24}(n)<1$
and for $n\rightarrow\infty$ we have that $\Theta_{24}(n) \rightarrow 1$.
\end{series}
\hrule

\begin{series}\label{ser25}
If we denote by $1/\pi_{25}(n)$ the truncated series
\begin{align*}
\frac{1}{\pi_{25}(n)} = \frac{205\sqrt{3}}{96}\sum_{k=0}^{n-1} \frac{\left(\frac{1}{2}\right)_k\left(\frac{1}{3}\right)_k\left(\frac{2}{3}\right)_k}{(k!)^3}\left(k+\frac{53}{615}\right)\left(\frac{-1}{1024}\right)^k,
\end{align*}
then we have that $\pi_{25}(n)\rightarrow\pi$ as $n\rightarrow\infty$ and we have the asymptotic expansion
\begin{dmath*}
\pi_{25}(n)-\pi=\left(\frac{-1}{1024}\right)^n\frac{16\sqrt{\pi}}{5\sqrt{n}}\cdot\exp\left(
\frac{-469}{1800n}
+\frac{-13}{3000n^2}
+\frac{3482081}{243000000n^3}
\Theta_{25}(n)\right).
\end{dmath*}
For all $n\geq 1$ we have that $0<\Theta_{25}(n)<1$
and for $n\rightarrow\infty$ we have that $\Theta_{25}(n) \rightarrow 1$.
\end{series}
\hrule

\begin{series}\label{ser26}
If we denote by $1/\pi_{26}(n)$ the truncated series
\begin{align*}
\frac{1}{\pi_{26}(n)} = \frac{4717\sqrt{3}}{1500}\sum_{k=0}^{n-1} \frac{\left(\frac{1}{2}\right)_k\left(\frac{1}{3}\right)_k\left(\frac{2}{3}\right)_k}{(k!)^3}\left(k+\frac{827}{14151}\right)\left(\frac{-1}{250000}\right)^k,
\end{align*}
then we have that $\pi_{26}(n)\rightarrow\pi$ as $n\rightarrow\infty$ and we have the asymptotic expansion
\begin{dmath*}
\pi_{26}(n)-\pi=\left(\frac{-1}{250000}\right)^n\frac{250\sqrt{\pi}}{53\sqrt{n}}\cdot\exp\left(
\frac{-58405}{202248n}
+\frac{-107963}{63123848n^2}
+\frac{4656308646401}{344700144278208n^3}
\Theta_{26}(n)\right).
\end{dmath*}
For all $n\geq 1$ we have that $0<\Theta_{26}(n)<1$
and for $n\rightarrow\infty$ we have that $\Theta_{26}(n) \rightarrow 1$.
\end{series}
\hrule

\begin{series}\label{ser27}
If we denote by $1/\pi_{27}(n)$ the truncated series
\begin{align*}
\frac{1}{\pi_{27}(n)} = \frac{5\sqrt{3}}{4}\sum_{k=0}^{n-1} \frac{\left(\frac{1}{2}\right)_k\left(\frac{1}{3}\right)_k\left(\frac{2}{3}\right)_k}{(k!)^3}\left(k+\frac{1}{5}\right)\left(\frac{-9}{16}\right)^k,
\end{align*}
then we have that $\pi_{27}(n)\rightarrow\pi$ as $n\rightarrow\infty$ and we have the asymptotic expansion
\begin{dmath*}
\pi_{27}(n)-\pi=\left(\frac{-9}{16}\right)^n\frac{6\sqrt{\pi}}{5\sqrt{n}}\cdot\exp\left(
\frac{59}{1800n}
+\frac{-127}{1000n^2}
+\frac{-186271}{243000000n^3}
+\frac{2740273}{25000000n^4}
\Theta_{27}(n)\right).
\end{dmath*}
For all $n\geq 31$ we have that $0<\Theta_{27}(n)<1$
and for $n\rightarrow\infty$ we have that $\Theta_{27}(n) \rightarrow 1$.
\end{series}
\hrule

\begin{series}\label{ser28}
If we denote by $1/\pi_{28}(n)$ the truncated series
\begin{align*}
\frac{1}{\pi_{28}(n)} = \frac{3\sqrt{15}}{4}\sum_{k=0}^{n-1} \frac{\left(\frac{1}{2}\right)_k\left(\frac{1}{3}\right)_k\left(\frac{2}{3}\right)_k}{(k!)^3}\left(k+\frac{1}{9}\right)\left(\frac{-1}{80}\right)^k,
\end{align*}
then we have that $\pi_{28}(n)\rightarrow\pi$ as $n\rightarrow\infty$ and we have the asymptotic expansion
\begin{dmath*}
\pi_{28}(n)-\pi=\left(\frac{-1}{80}\right)^n\frac{10\sqrt{5\pi}}{9\sqrt{n}}\cdot\exp\left(
\frac{-149}{648n}
+\frac{-715}{52488n^2}
+\frac{2207033}{102036672n^3}
+\frac{-20699503}{2754990144n^4}
+\frac{5625196531}{2231542016640n^5}
\Theta_{28}(n)\right).
\end{dmath*}
For all $n\geq 3$ we have that $0<\Theta_{28}(n)<1$
and for $n\rightarrow\infty$ we have that $\Theta_{28}(n) \rightarrow 1$.
\end{series}
\hrule

\begin{series}\label{ser29}
If we denote by $1/\pi_{29}(n)$ the truncated series
\begin{align*}
\frac{1}{\pi_{29}(n)} = \frac{55\sqrt{7}}{36}\sum_{k=0}^{n-1} \frac{\left(\frac{1}{2}\right)_k\left(\frac{1}{3}\right)_k\left(\frac{2}{3}\right)_k}{(k!)^3}\left(k+\frac{13}{165}\right)\left(\frac{-1}{3024}\right)^k,
\end{align*}
then we have that $\pi_{29}(n)\rightarrow\pi$ as $n\rightarrow\infty$ and we have the asymptotic expansion
\begin{dmath*}
\pi_{29}(n)-\pi=\left(\frac{-1}{3024}\right)^n\frac{42\sqrt{21\pi}}{55\sqrt{n}}\cdot\exp\left(
\frac{-58429}{217800n}
+\frac{-436999}{131769000n^2}
+\frac{5956214110601}{430489323000000n^3}
\Theta_{29}(n)\right).
\end{dmath*}
For all $n\geq 1$ we have that $0<\Theta_{29}(n)<1$
and for $n\rightarrow\infty$ we have that $\Theta_{29}(n) \rightarrow 1$.
\end{series}
\hrule

\begin{series}\label{ser30}
If we denote by $1/\pi_{30}(n)$ the truncated series
\begin{align*}
\frac{1}{\pi_{30}(n)} = \frac{2\sqrt{3}}{3}\sum_{k=0}^{n-1} \frac{\left(\frac{1}{2}\right)_k\left(\frac{1}{3}\right)_k\left(\frac{2}{3}\right)_k}{(k!)^3}\left(k+\frac{1}{6}\right)\left(\frac{1}{2}\right)^k,
\end{align*}
then we have that $\pi_{30}(n)\rightarrow\pi$ as $n\rightarrow\infty$ and we have the asymptotic expansion
\begin{dmath*}
\pi_{30}(n)-\pi=\left(\frac{1}{2}\right)^n\frac{2\sqrt{\pi}}{\sqrt{n}}\cdot\exp\left(
\frac{-49}{72n}
+\frac{7}{6n^2}
+\frac{-65647}{15552n^3}
+\frac{7057}{324n^4}
+\frac{-202977547}{1399680n^5}
\Theta_{30}(n)\right).
\end{dmath*}
For all $n\geq 1$ we have that $0<\Theta_{30}(n)<1$
and for $n\rightarrow\infty$ we have that $\Theta_{30}(n) \rightarrow 1$.
\end{series}
\hrule

\begin{series}\label{ser31}
If we denote by $1/\pi_{31}(n)$ the truncated series
\begin{align*}
\frac{1}{\pi_{31}(n)} = \frac{20}{9}\sum_{k=0}^{n-1} \frac{\left(\frac{1}{2}\right)_k\left(\frac{1}{3}\right)_k\left(\frac{2}{3}\right)_k}{(k!)^3}\left(k+\frac{2}{15}\right)\left(\frac{2}{27}\right)^k,
\end{align*}
then we have that $\pi_{31}(n)\rightarrow\pi$ as $n\rightarrow\infty$ and we have the asymptotic expansion
\begin{dmath*}
\pi_{31}(n)-\pi=\left(\frac{2}{27}\right)^n\frac{6\sqrt{3\pi}}{5\sqrt{n}}\cdot\exp\left(
\frac{-457}{1800n}
+\frac{19}{450n^2}
+\frac{-12156583}{243000000n^3}
+\frac{11619823}{126562500n^4}
+\frac{-2235858759091}{13668750000000n^5}
\Theta_{31}(n)\right).
\end{dmath*}
For all $n\geq 3$ we have that $0<\Theta_{31}(n)<1$
and for $n\rightarrow\infty$ we have that $\Theta_{31}(n) \rightarrow 1$.
\end{series}
\hrule

\begin{series}\label{ser32}
If we denote by $1/\pi_{32}(n)$ the truncated series
\begin{align*}
\frac{1}{\pi_{32}(n)} = \frac{22\sqrt{3}}{15}\sum_{k=0}^{n-1} \frac{\left(\frac{1}{2}\right)_k\left(\frac{1}{3}\right)_k\left(\frac{2}{3}\right)_k}{(k!)^3}\left(k+\frac{4}{33}\right)\left(\frac{4}{125}\right)^k,
\end{align*}
then we have that $\pi_{32}(n)\rightarrow\pi$ as $n\rightarrow\infty$ and we have the asymptotic expansion
\begin{dmath*}
\pi_{32}(n)-\pi=\left(\frac{4}{125}\right)^n\frac{25\sqrt{\pi}}{11\sqrt{n}}\cdot\exp\left(
\frac{-2113}{8712n}
+\frac{580}{43923n^2}
+\frac{-262753543}{27551316672n^3}
+\frac{491042788}{17363069361n^4}
+\frac{-138715488606809}{3300372224138880n^5}
\Theta_{32}(n)\right).
\end{dmath*}
For all $n\geq 3$ we have that $0<\Theta_{32}(n)<1$
and for $n\rightarrow\infty$ we have that $\Theta_{32}(n) \rightarrow 1$.
\end{series}
\hrule

\begin{series}\label{ser33}
If we denote by $1/\pi_{33}(n)$ the truncated series
\begin{align*}
\frac{1}{\pi_{33}(n)} = 2\sum_{k=0}^{n-1} \frac{\left(\frac{1}{2}\right)_k\left(\frac{1}{2}\right)_k\left(\frac{1}{2}\right)_k}{(k!)^3}\left(k+\frac{1}{4}\right)\left(-1\right)^k,
\end{align*}
then we have that $\pi_{33}(n)\rightarrow\pi$ as $n\rightarrow\infty$ and we have the asymptotic expansion
\begin{dmath*}
\frac{1}{\pi}-\frac{1}{\pi_{33}(n)}=\left(-1\right)^n\frac{1}{\sqrt{\pi^3 n}}\cdot\exp\left(
\frac{1}{8n}
+\frac{-1}{8n^2}
+\frac{-13}{192n^3}
+\frac{7}{64n^4}
+\frac{81}{640n^5}
\Theta_{33}(n)\right).
\end{dmath*}
For all $n\geq 2$ we have that $0<\Theta_{33}(n)<1$
and for $n\rightarrow\infty$ we have that $\Theta_{33}(n) \rightarrow 1$.
\end{series}
\hrule

\begin{series}\label{ser34}
If we denote by $1/\pi_{34}(n)$ the truncated series
\begin{align*}
\frac{1}{\pi_{34}(n)} = \frac{3\sqrt{2}}{2}\sum_{k=0}^{n-1} \frac{\left(\frac{1}{2}\right)_k\left(\frac{1}{2}\right)_k\left(\frac{1}{2}\right)_k}{(k!)^3}\left(k+\frac{1}{6}\right)\left(\frac{-1}{8}\right)^k,
\end{align*}
then we have that $\pi_{34}(n)\rightarrow\pi$ as $n\rightarrow\infty$ and we have the asymptotic expansion
\begin{dmath*}
\pi_{34}(n)-\pi=\left(\frac{-1}{8}\right)^n\frac{4\sqrt{2\pi}}{3\sqrt{n}}\cdot\exp\left(
\frac{-11}{72n}
+\frac{-23}{324n^2}
+\frac{8171}{139968n^3}
+\frac{-1823}{209952n^4}
+\frac{-1668011}{37791360n^5}
\Theta_{34}(n)\right).
\end{dmath*}
For all $n\geq 6$ we have that $0<\Theta_{34}(n)<1$
and for $n\rightarrow\infty$ we have that $\Theta_{34}(n) \rightarrow 1$.
\end{series}
\hrule

\begin{series}\label{ser35}
If we denote by $1/\pi_{35}(n)$ the truncated series
\begin{align*}
\frac{1}{\pi_{35}(n)} = \frac{3}{2}\sum_{k=0}^{n-1} \frac{\left(\frac{1}{2}\right)_k\left(\frac{1}{2}\right)_k\left(\frac{1}{2}\right)_k}{(k!)^3}\left(k+\frac{1}{6}\right)\left(\frac{1}{4}\right)^k,
\end{align*}
then we have that $\pi_{35}(n)\rightarrow\pi$ as $n\rightarrow\infty$ and we have the asymptotic expansion
\begin{dmath*}
\pi_{35}(n)-\pi=\left(\frac{1}{4}\right)^n\frac{2\sqrt{\pi}}{\sqrt{n}}\cdot\exp\left(
\frac{-3}{8n}
+\frac{1}{4n^2}
+\frac{-31}{64n^3}
+\frac{41}{32n^4}
+\frac{-2723}{640n^5}
\Theta_{35}(n)\right).
\end{dmath*}
For all $n\geq 4$ we have that $0<\Theta_{35}(n)<1$
and for $n\rightarrow\infty$ we have that $\Theta_{35}(n) \rightarrow 1$.
\end{series}
\hrule

\begin{series}\label{ser36}
If we denote by $1/\pi_{36}(n)$ the truncated series
\begin{align*}
\frac{1}{\pi_{36}(n)} = \frac{21}{8}\sum_{k=0}^{n-1} \frac{\left(\frac{1}{2}\right)_k\left(\frac{1}{2}\right)_k\left(\frac{1}{2}\right)_k}{(k!)^3}\left(k+\frac{5}{42}\right)\left(\frac{1}{64}\right)^k,
\end{align*}
then we have that $\pi_{36}(n)\rightarrow\pi$ as $n\rightarrow\infty$ and we have the asymptotic expansion
\begin{dmath*}
\pi_{36}(n)-\pi=\left(\frac{1}{64}\right)^n\frac{8\sqrt{\pi}}{3\sqrt{n}}\cdot\exp\left(
\frac{-19}{72n}
+\frac{1}{324n^2}
+\frac{667}{139968n^3}
+\frac{2689}{209952n^4}
+\frac{-753139}{37791360n^5}
\Theta_{36}(n)\right).
\end{dmath*}
For all $n\geq 2$ we have that $0<\Theta_{36}(n)<1$
and for $n\rightarrow\infty$ we have that $\Theta_{36}(n) \rightarrow 1$.
\end{series}

\end{document}